\RequirePackage{fix-cm}
\documentclass[smallextended]{svjour3}       
\smartqed  
\usepackage{amssymb}
\usepackage{graphicx}
\usepackage{enumerate}
\usepackage{amsmath}
\usepackage{arydshln}
\begin{document}

\title{High Order Finite Difference Methods on Non-uniform Meshes for Space Fractional Operators
\thanks{This work was supported by the National Natural Science
Foundation of China under Grant No. 11271173.}
}

\titlerunning{Finite Difference Methods on Non-uniform Meshes for Fractional Operators}        

\author{Lijing Zhao         \and
        Weihua Deng 
}


\institute{Lijing Zhao \at
              School of Mathematics and Statistics, Gansu Key Laboratory of Applied Mathematics and Complex Systems,
              Lanzhou University, Lanzhou 730000, P.R. China \\
              \email{zhaolj10@lzu.edu.cn}           
           \and
           Weihua Deng \at
              School of Mathematics and Statistics, Gansu Key Laboratory of Applied Mathematics and Complex Systems,
              Lanzhou University, Lanzhou 730000, P.R. China \\
               \email{dengwh@lzu.edu.cn, dengwhmath@aliyun.com}
}

\date{Received: date / Accepted: date}

\maketitle

\begin{abstract}

In the past decades, the finite difference methods for space fractional operators develop rapidly; to the best of our knowledge, all the existing finite difference schemes, including the first and high order ones, just work on uniform meshes. The nonlocal property of space fractional operator makes it difficult to design the finite difference scheme on non-uniform meshes. This paper provides a basic strategy to derive the first and high order discretization schemes on non-uniform meshes for fractional operators. And the obtained first and second schemes on non-uniform meshes are used to solve space fractional diffusion equations. The error estimates and stability analysis are detailedly performed; and extensive numerical experiments confirm the theoretical analysis or verify the convergence orders.


\end{abstract}

\section{Introduction}
\label{sec:1}
The continuous time random walk (CTRW) model plays a key role in statistical physics; and it describes the L{\'e}vy flight of anomalous diffusion when the first moment of waiting time distribution is finite and the distribution of the jump length is power law with divergent second moment. Based on this concrete CTRW model, the corresponding Fokker-Planck equation can be derived, which is the space fractional diffusion equation describing the time evolution of the probability density function of the L{\'e}vy-flight particles  \cite{Metzler:00}.


Over the last decades, the finite difference methods have achieved important developments in solving the
space fractional differential equations with Riemann-Liouville derivative, e.g., \cite{Celik:12,Chen2:13,Meerschaert:04,Yang:10}. The Riemann-Liouville space fractional derivative can be naturally discretized by the standard Gr\"{u}nwald formula \cite{Podlubny:99} with first order accuracy or Lubich's high order formulas \cite{Lubich:86}, but the finite difference schemes derived by the discretization are unconditionally unstable for the initial value problems including the implicit schemes that are well known
to be stable most of the time for classical derivatives \cite{Meerschaert:04}. To remedy the instability, Meerschaert and Tadjeran in \cite{Meerschaert:04} firstly propose the shifted Gr\"{u}nwald  formula to approximate fractional advection-dispersion flow equations with still first order accuracy. At that time, the shifted Gr\"{u}nwald discretization was the most effective and popular approximation to the space fractional operator.

Recently, the effective high order approximations to space fractional derivatives appeared \cite{Chen:13,Chen2:13,Li:13,Sousa:11,Tian:12,Zhao:14,Zhou:13}. Using the idea of second order central difference formula, a second order discretization for Riemann-Liouville fractional derivative is established in \cite{Sousa:11}, and the scheme is extended to two dimensional two-sided space fractional convection diffusion equation in finite domain in \cite{Chen:13}. By assembling the
Gr\"{u}nwald difference operators with different weights and shifts,  a class of stable second
order discretizations for Riemann-Liouville space fractional derivative is developed in \cite{Tian:12,Li:13}; and
its corresponding third order quasi-compact scheme is given in \cite{Zhou:13}.
Allowing to use the points outside of the domain, a class of second, third and fourth order approximations for Riemann-Liouville space fractional derivatives are derived in \cite{Chen2:13} by using the weighted and shifted Lubich difference operators.
More high order quasi-compact schemes based on the superconvergent approximations for fractional derivatives can be found in \cite{Zhao:14}.

Generally speaking, high order schemes lead to more accurate results, supposing that the solution of the equation to be solved is regular enough. In fact, compared with the classical partial differential equations (PDEs), the high order finite difference schemes obtain more striking benefits than the low order ones in solving the fractional PDEs. The reason is that the high order schemes can keep the same computational cost as the first order ones but greatly improve the accuracy. Usually, the high and low order schemes have the same algebraic structures, e.g., $(\textbf{T}-\textbf{A})\textbf{U}^{n+1}=(\textbf{T}+\textbf{A})\textbf{U}^{n}+\textbf{b}^{n+1}$, where $\textbf{T}$ is tri-diagonal, and $\textbf{A}$ is Toeplitz-like full matrix. Regarding to the possible weak regularity of the solution to the space fractional PDEs at the closest regions of the boundaries, the techniques introduced in the present paper and \cite{Chen:14,Zhao:14} can effectively treat it.


%

The techniques proposed in \cite{Chen:14,Zhao:14} deal with the boundary issue, essentially the issue of regularity of the solution in (extended) unbounded domain. More concretely, for obtaining the $l$-th convergence order, the solution of the equation must satisfy: $u(x)\in C^{l+\lfloor \alpha \rfloor-1}[a,b]$, $D^{l+\lfloor \alpha \rfloor}\in L^1[a,b]$, and $D^{k}u(a)=0,~k=0,1,\cdots,l+\lfloor \alpha \rfloor-2$, where $\alpha$ is the order of fractional derivative \cite{Tian:12,Zhou:13}.  The techniques in \cite{Chen:14,Zhao:14} remove the nonphysical boundary requirement but keep high convergence order; however, high regularity requirements on the whole domain $[a,b]$ are still needed. Generally, the regularity of the solution varies on the domain; in particular, for space fractional problem, the regularity of its solution at the region close to the boundary is much different from other places. As to that, the existing high order approaches cannot work effectively. This paper targets this issue by providing the numerical schemes on non-uniform meshes. Roughly speaking, similar to deal with the classical problems on non-uniform meshes, the basic idea here is that fine grids are used on the domain of weaker regularity, while coarse grids apply to the high regularity area. The nonlocal property of fractional operators makes it much complex or difficult to design the finite difference scheme on non-uniform meshes; currently, it seems there are no published works to deal with this topic.

Our idea of deriving the scheme is first to rewrite the function to be approximated as the sum of two functions having at least the same regularities as itself, then use different meshes to discretize the two functions; the common support of the functions are usually small. The obtained first and second order schemes on non-uniform meshes are used to solve the space fractional diffusion equation; and it should be noted that the interpolation is used in discretizing the equation. The error estimate and numerical stability of the schemes are detailedly discussed. And the extensive numerical experiments verify the theoretical results, including the convergence orders.

The organization of this paper is as follows. In Section
\ref{sec:2}, we introduce a kind of smooth functions, by which a function to be fractionally differentiated can be rendered into two other ones, which are much more manageable; in this way, two kinds of approximations on non-uniform meshes are discussed.
The derived schemes are used to solve a space fractional diffusion equation in Section \ref{sec:3}, and the convergence and stability analysis are performed in detail. Numerical results are given in Section \ref{sec:4}, which confirm the theoretical analysis and convergence orders. We conclude the paper with some remarks in Section \ref{sec:5}. For the concrete expressions of the matrix forms in Section \ref{sec:3}, refer to Appendix.


\section{Discretization on non-uniform meshes for fractional derivative}
\label{sec:2}

This section focuses on deriving the discretization methods on non-uniform meshes for the left Riemann-Liouville fractional
derivative; the ones for the right Riemann-Liouville fractional derivative or Riemann-Liouville fractional integral can be got by almost the same way, so they are omitted here.  The left Riemann-Liouville fractional derivative of a function $u(x)$ on $[a,b]$, $-\infty \leq a <b\leq \infty$ is defined by \cite{Podlubny:99,Samko:93}
\begin{equation}\label{equation1.1}
_{a}D_x^{\alpha}u(x)= D^{m}{}_{a}D_{x}^{-(m-\alpha)}u(x),
\end{equation}
where $\alpha \in (m-1,m)$ and
\begin{equation}\label{equation1.2}
{}_{a}D_{x}^{-\gamma}u(x)
=\frac{1}{\Gamma(\gamma)}\int_{a}\nolimits^x{\left(x-\xi\right)^{\gamma-1}}{u(\xi)}d\xi,  ~~~~~~ \gamma > 0,
\end{equation}
is the $\gamma$-th left Riemann-Liouville fractional integral.


For classical problem, designing the finite difference scheme on non-uniform meshes is an easy task. Is it still true for fractional operators? Let us see the following example.
Using the Gr\"{u}nwald  difference operator \cite{Meerschaert:02} as an approximation, if $u(x)\in \mathbf{C}^1[0,b]$, $D^2 u(x)\in L^1(0,b)$, and $u(0)=0$, then
\begin{eqnarray}\label{equation2.0.1}
&&\,_{0}D_{x}^{\alpha}u(x_{mn})
\nonumber\\
&=& h_1^{-\alpha}\sum_{k=0}^{n}g_k^{(\alpha)}u\left((n-k)h_1\right)+O(h_1)
\nonumber\\
&=& h_2^{-\alpha}\sum_{k=0}^{mn}g_k^{(\alpha)}u\left((m n-k)h_2\right)+O(h_2)
\end{eqnarray}
holds \cite{Zhao:14}, where $h_1=mh_2$, $x_k=kh_2$, $g_k^{(\alpha)}=(-1)^k\left( \begin{array}{c}\alpha \\ k\end{array} \right )$ are the coefficients of the power series of the generating function $(1-\zeta)^{\alpha}$;
 and they can be calculated by the following recursive formula
\begin{equation}\label{equation2.0.2}
g_0^{(\alpha)}=1, ~~~~g_k^{(\alpha)}=\left(1-\frac{\alpha+1}{k}\right)g_{k-1}^{\alpha},~~k \geq 1.
\end{equation}
For non-uniform grids, two naives ways are to replace the right side of (\ref{equation2.0.1}) by
\begin{eqnarray}\label{equation2.0.3}
\,_{0}D_{x}^{\alpha}u(x_{mn})&\approx& \Phi_{mn}
\nonumber\\
&:=& h_2^{-\alpha}\sum_{k=0}^{mn-[\frac{a}{h_2}]-1}g_k^{(\alpha)}u\left((mn-k)h_2\right)
\nonumber\\
&&+h_1^{-\alpha}\sum_{k=n-[\frac{a}{h_1}]}^{n}g_k^{(\alpha)}u\left((n-k)h_1\right),
\end{eqnarray}
or
\begin{eqnarray}\label{equation2.0.3_1}
\,_{0}D_{x}^{\alpha}u(x_{mn})&\approx& \Psi_{mn}
\nonumber\\
&:=& h_1^{-\alpha}\sum_{k=0}^{n-[\frac{a}{h_1}]-1}g_k^{(\alpha)}u\left((n-k)h_1\right)
\nonumber\\
&&+h_2^{-\alpha}\sum_{k=mn-[\frac{a}{h_2}]}^{mn}g_k^{(\alpha)}u\left((mn-k)h_2\right),
\end{eqnarray}
where $0<a<b$ is the dividing point. Figure \ref{fig:2.1.1} illustrates the non-uniform
grids of considering discretization formula of Eq. (\ref{equation2.0.3}). However, neither of these two schemes don't work.
By taking $u=x^{1+\alpha}$, $a=1/3$, $b=1$, and $h_2=h_1/2$, we can see from Tables \ref{table2.0.1} and  \ref{table2.0.1_1} that
the discrete $L^2$ errors $e=\left\|\,_{0}D_{x}^{\alpha}u-\Phi\right\|$ w.r.t. the stepsize $h_1$ do not decreace with the reducing of $h_1$ and $h_2$.

\begin{figure}[!htbp]
\includegraphics[scale=0.28]{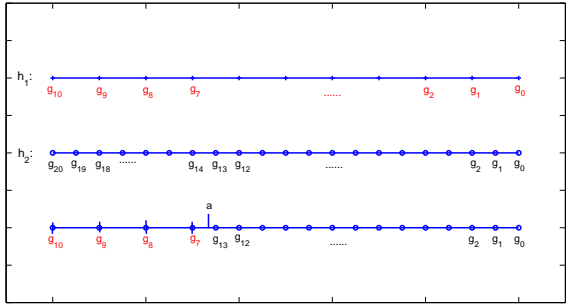}\\
\caption{The non-uniform grids illustrating the intuition of considering the formula of Eq. (\ref{equation2.0.3}) where $h_2=h_1/2$}\label{fig:2.1.1}
\end{figure}

\begin{table}
\caption{The discrete $L^2$ errors of (\ref{equation2.0.3}) w.r.t. the stepsize $h_1$ for $u(x)=x^{1+\alpha}$ with
different $\alpha$ and $h_1$, where $a=1/3$, $b=1$, and $h_2=h_1/2$.}\label{table2.0.1}
\begin{tabular}{ccccccccc}
\hline\noalign{\smallskip}
$\alpha$   &\multicolumn{2}{c}{$1.2$} &\multicolumn{2}{c}{$1.4$} &\multicolumn{2}{c}{$1.6$} &\multicolumn{2}{c}{$1.8$}\\
\noalign{\smallskip}\hline\noalign{\smallskip}
$h_1$& $e$& rate & $e$ & rate & $e$ & rate & $e$ & rate \\
\noalign{\smallskip}\hline\noalign{\smallskip}
1/10 &5.02 1e-1 &-     &9.16 1e-1 &-     &1.63 1e+0 &-     &2.82 1e+0 & -   \\
1/20 &8.42 1e-1 &-0.75 &1.74 1e+0 &-0.92 &3.49 1e+0 &-1.10 &6.87 1e+0 &-1.28 \\
1/40 &1.64 1e+0 &-0.96 &3.91 1e+0 &-1.17 &9.08 1e+0 &-1.38 &2.06 1e+1 &-1.59 \\
1/80 &2.66 1e+0 &-0.70 &7.25 1e+0 &-0.89 &1.93 1e+1 &-1.09 &5.01 1e+1 &-1.28 \\
1/160&4.50 1e+0 &-0.76 &1.41 1e+1 &-0.96 &4.31 1e+1 &-1.16 &1.29 1e+2 &-1.37 \\
\noalign{\smallskip}\hline
\end{tabular}
\end{table}

\begin{table}
\caption{The discrete $L^2$ errors of (\ref{equation2.0.3_1}) w.r.t. the stepsize $h_1$ for $u(x)=x^{1+\alpha}$ with
different $\alpha$ and $h_1$, where $a=1/3$, $b=1$, and $h_2=h_1/2$.}\label{table2.0.1_1}
\begin{tabular}{ccccccccc}
\hline\noalign{\smallskip}
$\alpha$   &\multicolumn{2}{c}{$1.2$} &\multicolumn{2}{c}{$1.4$} &\multicolumn{2}{c}{$1.6$} &\multicolumn{2}{c}{$1.8$}\\
\noalign{\smallskip}\hline\noalign{\smallskip}
$h_1$& $e$& rate & $e$ & rate & $e$ & rate & $e$ & rate \\
\noalign{\smallskip}\hline\noalign{\smallskip}
1/10 &5.59 1e-1 &-     &9.95 1e-1 &-     &1.74 1e+0 &-     &2.95 1e+0 & -   \\
1/20 &1.03 1e+0 &-0.89 &1.93 1e+0 &-0.96 &3.56 1e+0 &-1.03 &6.45 1e+0 &-1.11 \\
1/40 &1.65 1e+0 &-0.67 &3.92 1e+0 &-1.02 &9.09 1e+0 &-1.36 &2.07 1e+1 &-1.68 \\
1/80 &2.55 1e+0 &-0.63 &6.24 1e+0 &-0.67 &1.50 1e+1 &-0.72 &3.59 1e+1 &-0.80 \\
1/160&4.50 1e+0 &-0.82 &1.41 1e+1 &-1.18 &4.31 1e+1 &-1.52 &1.29 1e+2 &-1.85 \\
\noalign{\smallskip}\hline
\end{tabular}
\end{table}

Now we dig out the reason of non-convergence of (\ref{equation2.0.3}). In fact, if we rewrite $u(x)$ as the sum of $\tilde{u}_1(x)$ and $\tilde{u}_2(x)$, where
\begin{equation}\label{equation2.0.4}
\tilde{u}_1(x)=\left\{ \begin{array}{ll}
u(x),& 0<x\leq a,\\
0,& a<x\leq b,\\
\end{array} \right.
~~~~~
\tilde{u}_2(x)=\left\{ \begin{array}{ll}
0,& 0<x\leq a,\\
u(x),& a<x\leq b,\\
\end{array} \right.
\end{equation}
and approximate $\tilde{u}_1(x)$ and $\tilde{u}_2(x)$ by the Gr\"{u}nwald formulae with stepsize $h_1$ and $h_2$, respectively, then the discretization scheme (\ref{equation2.0.3}) is obtained. However, it can be noticed that both $\tilde{u}_1(x)$ and $\tilde{u}_2(x)$ are not even continuous, which violates the regularity requirements, being given in (\ref{equation2.0.1}), for the function to be approximated.


\begin{remark}
It should be clarified that for the errors $e=\left|\,_{0}D_{x}^{\alpha}u-\Phi\right|$ at one point
$x=b$, we can observe their convergence behaviors. See Table \ref{table2.0.2}. The following analysis illustrates the reason.

Since
\begin{equation}
g_k^{(\alpha)}=(-1)^k\left( \begin{array}{c}\alpha \\ k\end{array} \right )=\frac{\Gamma(k-\alpha)}{\Gamma(k+1)},
\end{equation}
by using Stiring's formula $\Gamma(x+1)\sim (2\pi)^{\frac{1}{2}}x^{x+\frac{1}{2}}e^{-x}$ as $x\rightarrow \infty$, there is
$g_k^{(\alpha)}\sim k^{-\alpha-1}$
as $k\rightarrow \infty$. Then
\begin{eqnarray}
&&\,_{0}D_{x}^{\alpha}u(x_{mn})-\Phi_{mn}
\nonumber\\
&=& h_2^{-\alpha}\sum_{k=0}^{mn}g_k^{(\alpha)}u\left((mn-k)h_2\right)+O(h_2)
\nonumber\\
&&-h_2^{-\alpha}\sum_{k=0}^{mn-[\frac{a}{h_2}]-1}g_k^{(\alpha)}u\left((mn-k)h_2\right)
-h_1^{-\alpha}\sum_{k=n-[\frac{a}{h_1}]}^{n}g_k^{(\alpha)}u\left((n-k)h_1\right)
\nonumber\\
&=&h_2^{-\alpha}\sum_{k=mn-[\frac{a}{h_2}]}^{mn}g_k^{(\alpha)}u\left((mn-k)h_2\right)
-h_1^{-\alpha}\sum_{k=n-[\frac{a}{h_1}]}^{n}g_k^{(\alpha)}u\left((n-k)h_1\right)
\nonumber\\
&& +O(h_2)
\nonumber\\
&=&h_2^{-\alpha}\sum_{k=n-[\frac{a}{h_1}]}^{n}
\left[ \sum_{q=0}^{m-1}g_{km+q}^{(\alpha)}u\left(((n-k)h_1-q h_2)\right)- m^{-\alpha} g_k^{(\alpha)}u\left((n-k)h_1\right)\right]
\nonumber\\
&& +O(h_2)
\nonumber\\
&\sim&h_2^{-\alpha}\sum_{k=n-[\frac{a}{h_1}]}^{n}
\Bigg[ \sum_{q=0}^{m-1}(km+q)^{-\alpha-1}u\left(((n-k)h_1-q h_2)\right)
\nonumber\\
&&
~~~~~~~~~~~~~~~~~~~~-(mk)^{-\alpha} k^{-1}u\left((n-k)h_1\right)\Bigg]+O(h_2) \nonumber
\end{eqnarray}
\begin{eqnarray}
&\sim&\sum_{k=n-[\frac{a}{h_1}]}^{n}(k h_1)^{-\alpha} k^{-1}
\left[m^{-1} \sum_{q=0}^{m-1}u\left(((n-k)h_1-q h_2)\right)-u\left((n-k)h_1\right)\right]
\nonumber\\
&& +O(h_2)
\nonumber\\
&\sim&\sum_{k=n-[\frac{a}{h_1}]}^{n}k^{-1}O(h_2)+O(h_2)
\sim O(h_2),
\end{eqnarray}
for large $k$. This means that if the dividing point $a$ is far away from the node $x_{mn}$ at which we approximate $\,_{0}D_{x}^{\alpha}u(x_{mn})$ by (\ref{equation2.0.3}), then the error $e=\left|\,_{0}D_{x}^{\alpha}u(x_{mn})-\Phi_{mn}\right|$
might have the first convergence order. However, this explanation does not hold when $k$ is small. That is the reason why the discrete $L^2$ error does not converge with the decreasing of $h_1$ and $h_2$ from another point of view.

\begin{table}
\caption{The errors $e=\left|\,_{0}D_{x}^{\alpha}u-\Phi\right|$ at $x=1$
by using the approximation (\ref{equation2.0.3}) for $u(x)=x^{1+\alpha}$ with different $\alpha$
and $h_1$, where $a=1/3$, and $h_2=h_1/2$.}\label{table2.0.2}
\begin{tabular}{ccccccccc}
\hline\noalign{\smallskip}
$\alpha$   &\multicolumn{2}{c}{$1.2$} &\multicolumn{2}{c}{$1.4$} &\multicolumn{2}{c}{$1.6$} &\multicolumn{2}{c}{$1.8$}\\
\noalign{\smallskip}\hline\noalign{\smallskip}
$h_1$& $e$& rate & $e$ & rate & $e$ & rate & $e$ & rate\\
\noalign{\smallskip}\hline\noalign{\smallskip}
1/10 &7.11 1e-2 &-    &1.02 1e-1 &-    &1.46 1e-1 &-    &2.09 1e-1 & -   \\
1/20 &3.68 1e-2 &0.95 &5.28 1e-2 &0.94 &7.50 1e-2 &0.96 &1.06 1e-1 & 0.98\\
1/40 &1.78 1e-2 &1.05 &2.55 1e-2 &1.05 &3.65 1e-2 &1.04 &5.24 1e-2 & 1.02\\
1/80 &9.19 1e-3 &0.95 &1.32 1e-2 &0.95 &1.87 1e-2 &0.96 &2.65 1e-2 & 0.98\\
1/160&4.45 1e-3 &1.05 &6.37 1e-3 &1.05 &9.13 1e-3 &1.04 &1.31 1e-2 & 1.02 \\
\noalign{\smallskip}\hline
\end{tabular}
\end{table}

\end{remark}

For overcoming the above challenge, a natural idea is to smooth the functions $\tilde{u}_1(x)$ and $\tilde{u}_2(x)$, such that they possess the regularities as well as their predecessor $u(x)$. Firstly, we introduce some properties of the mollificator, which plays a key role in our further discussion.


\subsection{Mollification}\label{subsec:2.1}

In this subsection, we introduce definitions on some special one-dimensional functions, such as smooth function, characteristic function and smooth truncated function, and some properties of these functions.

\begin{definition}(\cite{Evans:10}) \label{defi:2.1.1}
Define $\rho(x)\in \mathbf{C}^{\infty}(\mathbb{R})$ by
\begin{equation}\label{equation2.1.1}
\rho(x):=\left\{ \begin{array}{ll}
C^{-1}e^{\frac{1}{x^2-1}} &~~~{\rm if}~~ |x|<1\\
0 &~~~{\rm if}~~ |x|\geq 1,
\end{array} \right.
\end{equation}
where the constant $C$ is selected as $\int_{-1}^1 e^{\frac{1}{x^2-1}}$, so that
\begin{equation}\label{equation2.1.2}
\int_{\mathbb{R}}\rho(x) dx=\int_{-1}^{1}\rho(x)dx=1.
\end{equation}
\end{definition}

By (\ref{equation2.1.2}) and because of $\rho(x)$ being an even function, i.e., $\rho(x)=\rho(-x)$ for all $x\in \mathbb{R}$, there is
\begin{eqnarray}\label{equation2.1.add}
\int_{-1}^{x}\rho(z)dz&=&\int_{-x}^1\rho(z)dz
\nonumber\\
&=&\int_{-1}^{1}\rho(z)dz-\int_{-1}^{-x}\rho(z)dz
\nonumber\\
&=&1-\int_{-1}^{-x}\rho(z)dz.
\end{eqnarray}

\begin{definition}(\cite{Evans:10}) \label{defi:2.1.2}
If $f:U \rightarrow \mathbb{R}$ is locally integrable, where $U \subset \mathbb{R}$ is open, define its mollification
\begin{equation}\label{equation2.1.3}
(J_\epsilon f)(x):=\int_{-1}^1 \rho(z)f(x-\epsilon z)dz
=\int_{-1}^1 \rho(z)f(x+\epsilon z)dz ~~~{\rm in~~} U_{\epsilon},
\end{equation}
where $U_{\epsilon}=\{x\in U|{\rm dist}(x,\partial U)>\epsilon\}$. The second equality of (\ref{equation2.1.3}) holds because of $\rho(x)$ being an even function. It is well known (see such as Appendix C.4 in \cite{Evans:10}) that $(J_\epsilon f)(x)\in C^{\infty}(U_\epsilon)$.
\end{definition}

If $\Omega \subseteq \mathbb{R}$ is open, $\epsilon>0$, write $\Omega_1:=\{x|{\rm dist}(x, \Omega)<\epsilon\}$, $\Omega_2:=\{x|{\rm dist}(x, \Omega)<2\epsilon\}$.

\begin{definition}(\cite{Evans:10})\label{defi:2.1.3}
Define the characteristic function of $\Omega_1$ as
\begin{equation}\label{equation2.1.4}
\chi_{\substack{\\\Omega_1}}(x):=\left\{ \begin{array}{ll}
1 &~~ x \in \Omega_1\\
0 &~~ x \in \mathbb{R}\backslash \Omega_1;
\end{array} \right.
\end{equation}
and define its mollification $\eta(x):=J_{\epsilon}\chi_{\substack{\\\Omega_1}}(x)$, naming as the smooth truncated function of $\Omega$.
\end{definition}

By the above definitions, we can see that the truncated function has the following properties.

\begin{lemma}\label{lemma:2.1.1}
The smooth truncated function of $\Omega$ belongs to $C_{0}^{\infty}(\mathbb{R})$, and
\begin{equation}\label{equation2.1.5}
\eta(x)=\left\{ \begin{array}{ll}
\int_{-1}^{1-d(x)/\epsilon} \rho(z)dz &~~~x\in \Omega_2 \\
0 &~~~x\in \mathbb{R}\backslash \Omega_2,
\end{array} \right.
\end{equation}
where $d(x)={\rm dist}(x, \Omega)$.
\end{lemma}

\begin{proof}
Without loss of generality, we take $\Omega=(a,b)$.
By the above definitions,
\begin{equation}\label{equation2.1.6}
\eta(x)=\int_{-1}^1 \rho(z) \chi_{\substack{\\\Omega_1}}(x-\epsilon z)dz
=\int_{-1}^1 \rho(z) \chi_{\substack{\\\Omega_1}}(x+\epsilon z)dz.
\end{equation}
\begin{enumerate}[a)]
\item
If $x\in \Omega$, then $x+\epsilon z \in \Omega_1$ for $|z|\leq 1$, and $\chi_{\substack{\\\Omega_1}}(x+\epsilon z)=1$. Thus, by (\ref{equation2.1.6}) and (\ref{equation2.1.2}),
\begin{equation}
\eta(x)=\int_{-1}^1 \rho(z)dz=\int_{-1}^{1-d(x)/\epsilon} \rho(z)dz=1.
\end{equation}

\item
If $x\in \Omega_2 \backslash \Omega$ and  $x=b+r\epsilon$, $0\leq r<2$,
then $x+\epsilon z=b+(r+z)\epsilon$. When $r+z<1$, i.e., $z<1-r$, we have $x+\epsilon z \in \Omega_1$, and $\chi_{\substack{\\\Omega_1}}(x+\epsilon z)=1$; else, $\chi_{\substack{\\\Omega_1}}(x+\epsilon z)=0$.
Thus, by (\ref{equation2.1.6}),
\begin{equation}
\eta(x)=\int_{-1}^{1-r} \rho(z)dz=\int_{-1}^{1-\frac{x-b}{\epsilon}} \rho(z)dz
=\int_{-1}^{1-d(x)/\epsilon} \rho(z)dz.
\end{equation}

\item
If $x\in \Omega_2 \backslash \Omega$ and $x=a-r\epsilon$, $0\leq r<2$,
then $x-\epsilon z=a-(r+z)\epsilon$. When $r+z<1$, i.e., $z<1-r$, we have $x-\epsilon z \in \Omega_1$, and $\chi_{\substack{\\\Omega_1}}(x-\epsilon z)=1$; else, $\chi_{\substack{\\\Omega_1}}(x-\epsilon z)=0$.
Thus, by (\ref{equation2.1.6}),
\begin{equation}
\eta(x)=\int_{-1}^{1-r} \rho(z)dz=\int_{-1}^{1-\frac{a-x}{\epsilon}} \rho(z)dz
=\int_{-1}^{1-d(x)/\epsilon} \rho(z)dz.
\end{equation}

\item
If $x\in \mathbb{R}\backslash \Omega_2$, then $x+\epsilon z \notin \Omega_1$ for $|z|\leq 1$, and $\chi_{\substack{\\\Omega_1}}(x+\epsilon z)=0$. Thus, by (\ref{equation2.1.6}), $\eta(x)=0$.
\end{enumerate}
\end{proof}


Given a function $v(x):[0,b]\rightarrow \mathbb{R}$ for some $b>0$, and without ambiguity, after zero extending its definition to $\mathbb{R}$, we still denote the function as $v(x)$.

For fixed values $a>0$, $\epsilon>0$, with $a+2\epsilon<b$. Let $\eta_1(x)$ be the smooth truncated function of open set $(0,a)$, and $\eta_2(x)$ be the one of $(a+2\epsilon,b)$.
By Lemma \ref{lemma:2.1.1}, it is easy to see that
\begin{equation}\label{equation2.2.1}
v_1(x):=\eta_1(x)v(x)=\left\{ \begin{array}{ll}
v(x)&~~x\in(0,a)\\
v(x)\int_{-1}^{1-\frac{x-a}{\epsilon}}\rho(z)dz &~~ x \in [a,a+2\epsilon)\\
0 &~~ {\rm else};
\end{array} \right.
\end{equation}

\begin{equation}\label{equation2.2.2}
v_2(x):=\eta_2(x)v(x)=\left\{ \begin{array}{ll}
v(x)\int_{-1}^{\frac{x-a}{\epsilon}-1}\rho(z)dz&~~x\in(a,a+2\epsilon]\\
v(x) &~~ x \in (a+2\epsilon,b)\\
0 &~~ {\rm else}.
\end{array} \right.
\end{equation}

For the convenience of the following presentation, we define
the smooth operators $\mathcal{M}_j^{a,\epsilon}:
\mathbf{C}(\Omega)\rightarrow \mathbf{C}(\Omega)$, $j=1,2$, as
\begin{eqnarray}\label{equation2.2.3}
\mathcal{M}_j^{a,\epsilon}v(x):=v_j(x) ~~~\forall v\in \mathbf{C}(\Omega),~j=1,2,
\end{eqnarray}
where $v_j(x)$ are defined in (\ref{equation2.2.1}) and (\ref{equation2.2.2}), $\mathbf{C}(\Omega)$ denotes the space of  continuous functions defined on $\Omega$.
Obviously, there is
\begin{eqnarray}
\mathcal{M}_j^{a,\epsilon}\left(c v\right)(x)
&=&c\mathcal{M}_j^{a,\epsilon}v(x),
~~~\forall c\in \mathbb{R}, v\in \mathbf{C}(\Omega),~j=1,2;\label{equation2.2.5}
\\
\mathcal{M}_j^{a,\epsilon}\left(u+v\right)(x)&=&
\mathcal{M}_j^{a,\epsilon}u(x)+\mathcal{M}_j^{a,\epsilon}v(x), ~~~\forall u,v\in \mathbf{C}(\Omega),~j=1,2.\label{equation2.2.4}
\end{eqnarray}

\begin{lemma}\label{lemma:2.2.1}
If $v(x)\in \mathbf{C}(\Omega)$, $\Omega \subset \mathbb{R}$, then $\mathcal{M}_1^{a,\epsilon}v(x)+\mathcal{M}_2^{a,\epsilon}v(x)=v(x)$ for all $x\in \Omega$, and both regularities of $\mathcal{M}_1^{a,\epsilon}v(x)$ and $\mathcal{M}_2^{a,\epsilon}v(x)$ are as high as the one of $v(x)$.
\end{lemma}
\begin{proof}
By Eq. (\ref{equation2.1.add}) and the definitions of $v_1(x)$ and $v_2(x)$ in (\ref{equation2.2.1}) and (\ref{equation2.2.2}), it is clear that $v_1(x)+v_2(x)=v(x)$.
From Lemma \ref{lemma:2.1.1}, the smooth truncated functions $\eta_1(x)$ and $\eta_2(x)$ belong to $C_0^{\infty}(\mathbb{R})$, thus the regularities of $v_1(x)$ as well as $v_2(x)$ are no lower than the one of $v(x)$ for all $x\in \mathbb{R}$.

After restricting the domain onto $\Omega$, it is clear by (\ref{equation2.2.3}) that $\mathcal{M}_1^{a,\epsilon}v(x)+\mathcal{M}_2^{a,\epsilon}v(x)=v(x)$ for all $x\in \Omega$, also
the regularities of $\mathcal{M}_j^{a,\epsilon}v(x)$, $j=1,2$ are as high as the one of $v(x)$.
\end{proof}

\begin{remark}\label{remark:2.2.1}
In real computation, the integration $\int_{-1}^{1-\frac{x-a}{\epsilon}}\rho(z)dz$ for any
given $x$ is numerically calculated by the Gauss quadrature with 32 Gauss nodes.
Experiments show that $\max_{0\leq x \leq b}|\hat{v}_1(x)+\hat{v}_2(x)-v(x)|=C\cdot 10^{-8}$, where $\hat{v}_1(x)$ and
$\hat{v}_2(x)$ are respectively the approximations of $v_1(x)$ and $v_2(x)$ numerically calculated by this Gauss quadrature.
\end{remark}

Figures \ref{fig:2.2.1} and \ref{fig:2.2.2} are illustrations of these two smooth truncations $\mathcal{M}_1^{a,\epsilon}u(x)$ and $\mathcal{M}_2^{a,\epsilon}u(x)$ for function $u(x)=x^{1.5}$, comparing to functions $\tilde{u}_1(x)$ and $\tilde{u}_2(x)$ defined in (\ref{equation2.0.4}), with $a=0.4$, $b=1$, and $\epsilon=0.1$.
\begin{figure}[!htbp]
\includegraphics[scale=0.4]{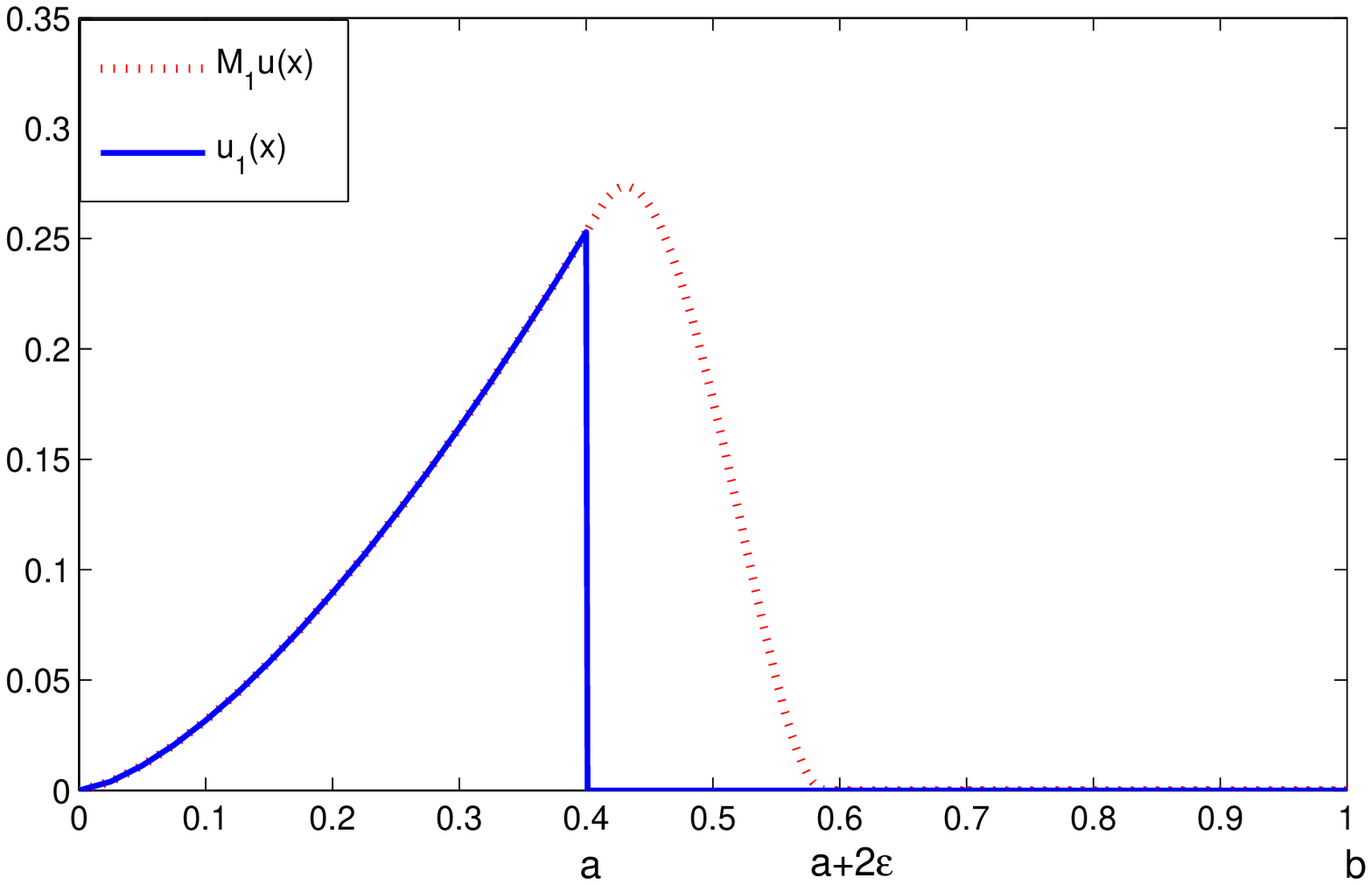}\\
\caption{The comparison between $\mathcal{M}_1^{a,\epsilon}u(x)$ (appeared as $M_1u(x)$ in the legend) and $\tilde{u}_1(x)$ (appeared as ${u}_1(x)$ in the legend) for $u(x)=x^{1.5}$, where $a=0.4$, $b=1$, and $\epsilon=0.1$.}\label{fig:2.2.1}
\includegraphics[scale=0.4]{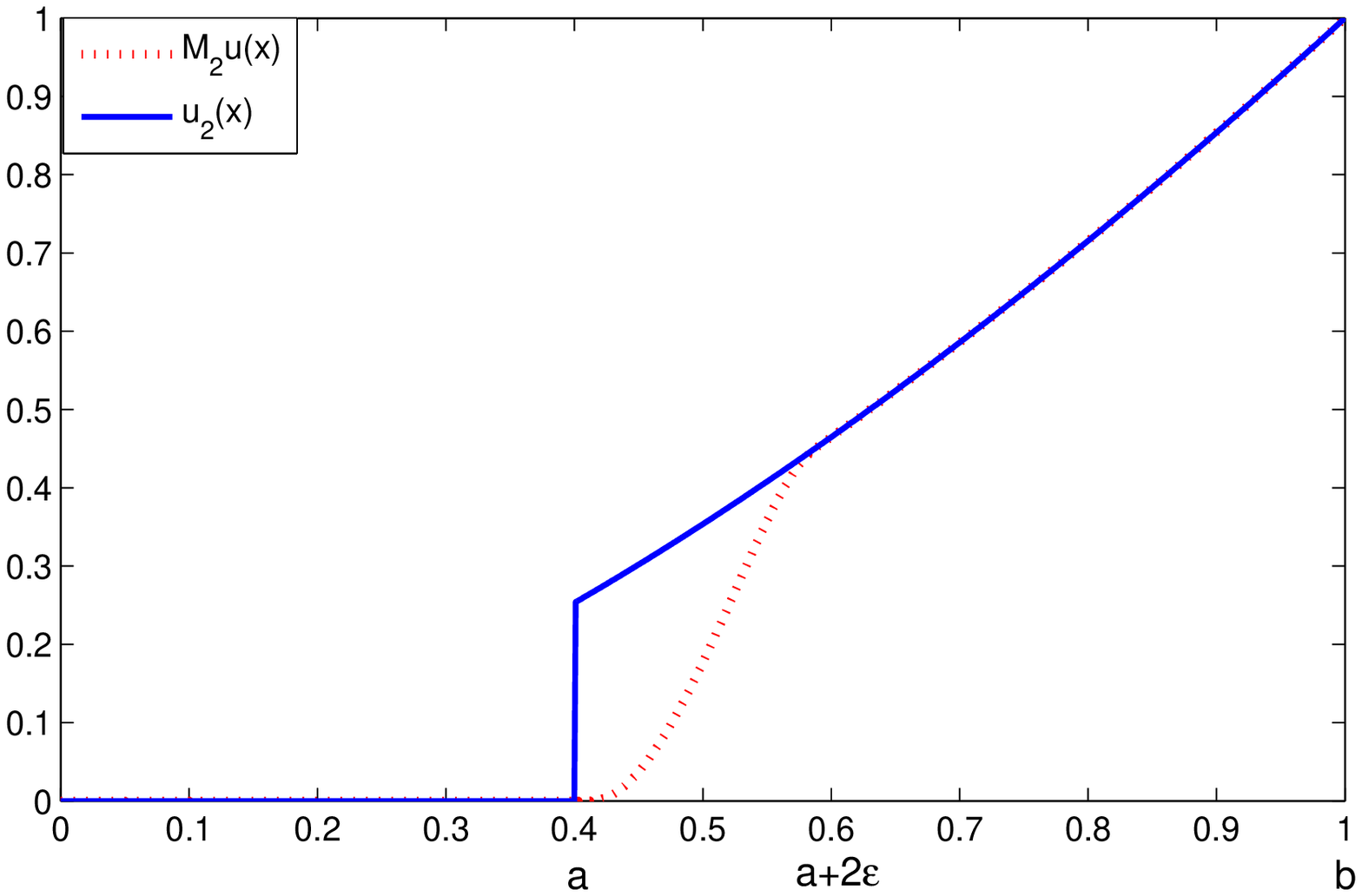}\\
\caption{The comparison between $\mathcal{M}_2^{a,\epsilon}u(x)$ (appeared as $M_2u(x)$ in the legend) and $\tilde{u}_2(x)$ (appeared as ${u}_2(x)$ in the legend) for $u(x)=x^{1.5}$, where $a=0.4$, $b=1$, and $\epsilon=0.1$.}\label{fig:2.2.2}
\end{figure}

Compared with (\ref{equation2.0.3}) or (\ref{equation2.0.3_1}), currently, we can easily get the following convergent discretization
\begin{eqnarray}\label{equation2.2.6}
&&{}_{0}D_{x}^{\alpha}u(x_{mn})
\nonumber\\
&=&{}_{0}D_{x}^{\alpha}u_1(x_{mn})+{}_{0}D_{x}^{\alpha}u_2(x_{mn})
\nonumber\\
&=&h_1^{-\alpha}\sum_{k=0}^{n}g_k^{(\alpha)} u_1((n-k)h_1)+O(h_1)
\nonumber\\
&&+h_2^{-\alpha}\sum_{k=0}^{mn}g_k^{(\alpha)} u_2((mn-k)h_2)+O(h_2),
\end{eqnarray}
or
\begin{eqnarray}\label{equation2.2.6_1}
&&{}_{0}D_{x}^{\alpha}u(x_{mn})
\nonumber\\
&=&{}_{0}D_{x}^{\alpha}u_1(x_{mn})+{}_{0}D_{x}^{\alpha}u_2(x_{mn})
\nonumber\\
&=&h_2^{-\alpha}\sum_{k=0}^{mn}g_k^{(\alpha)} u_1((mn-k)h_2)+O(h_2)
\nonumber\\
&&+h_1^{-\alpha}\sum_{k=0}^{n}g_k^{(\alpha)} u_2((n-k)h_1)+O(h_1).
\end{eqnarray}

More specifically, by Proposition 3.1 of \cite{Tadjeran:06} or Theorem 1 in \cite{Zhao:14},
there is
\begin{eqnarray}
&&{}_{0}D_{x}^{\alpha}u(x_{mn})
\nonumber\\
&=&{}_{0}D_{x}^{\alpha}u_1(x_{mn})+{}_{0}D_{x}^{\alpha}u_2(x_{mn})
\nonumber\\
&=&h_1^{-\alpha}\sum_{k=0}^{n}g_k^{(\alpha)} u_1((n-k)h_1)+
\frac{\alpha}{2}{}_{0}D_{x}^{\alpha+1}u_1(x_{mn})\cdot h_1
\nonumber\\
&&+h_2^{-\alpha}\sum_{k=0}^{mn}g_k^{(\alpha)} u_2((mn-k)h_2)
+\frac{\alpha}{2} {}_{0}D_{x}^{\alpha+1}u_2(x_{mn})\cdot h_2
\nonumber\\
&&+O(h_1^2+h_2^2),
\end{eqnarray}
or
\begin{eqnarray}
&&{}_{0}D_{x}^{\alpha}u(x_{mn})
\nonumber\\
&=&{}_{0}D_{x}^{\alpha}u_1(x_{mn})+{}_{0}D_{x}^{\alpha}u_2(x_{mn})
\nonumber\\
&=&h_2^{-\alpha}\sum_{k=0}^{mn}g_k^{(\alpha)} u_1((mn-k)h_2)+
\frac{\alpha}{2}{}_{0}D_{x}^{\alpha+1}u_1(x_{mn})\cdot h_2
\nonumber\\
&&+h_1^{-\alpha}\sum_{k=0}^{mn}g_k^{(\alpha)} u_2((n-k)h_1)
+\frac{\alpha}{2} {}_{0}D_{x}^{\alpha+1}u_2(x_{mn})\cdot h_1
\nonumber\\
&&+O(h_1^2+h_2^2).
\end{eqnarray}

Tables \ref{table2.2.1} and \ref{table2.2.1_1} shows the numerical discrete $L^2$ errors of (\ref{equation2.2.6}) and (\ref{equation2.2.6_1}) w.r.t. the stepsize $h_1$
for $u=x^{1+\alpha}$ and different $\alpha$, $h_1$, by taking $a=\epsilon=1$, $b=4$, and $h_2=h_1/2$.
It is easy to see that the mollification process does contribute to the approximation.

\begin{table}
\caption{The discrete $L^2$ errors, denoted as $e$, of (\ref{equation2.2.6}) for different $\alpha$ and $h_1$, with $a=\epsilon=1$, $b=4$, $h_2=h_1/2$, and $u(x)=x^{1+\alpha}$.}\label{table2.2.1}
\begin{tabular}{ccccccccc}
\hline\noalign{\smallskip}
$\alpha$   &\multicolumn{2}{c}{$1.2$} &\multicolumn{2}{c}{$1.4$} &\multicolumn{2}{c}{$1.6$} &\multicolumn{2}{c}{$1.8$}\\
\noalign{\smallskip}\hline\noalign{\smallskip}
$h_1$& $e$& rate & $e$ & rate & $e$ & rate & $e$ & rate\\
\noalign{\smallskip}\hline\noalign{\smallskip}
1/10 &2.52 1e-1 &-    &3.03 1e-1 &-    &3.57 1e-1 &-    &3.98 1e-1 & -   \\
1/20 &1.26 1e-1 &1.00 &1.51 1e-1 &1.00 &1.76 1e-1 &1.01 &1.95 1e-1 & 1.03\\
1/40 &6.33 1e-2 &1.00 &7.56 1e-2 &1.00 &8.63 1e-2 &1.03 &8.38 1e-2 & 1.22\\
1/80 &3.17 1e-2 &1.00 &3.78 1e-2 &1.00 &4.30 1e-2 &1.01 &3.98 1e-2 & 1.08\\
1/160&1.59 1e-2 &1.00 &1.89 1e-2 &1.00 &2.15 1e-2 &1.00 &1.96 1e-2 & 1.02 \\
\noalign{\smallskip}\hline
\end{tabular}
\end{table}

\begin{table}
\caption{The discrete $L^2$ errors, denoted as $e$, of (\ref{equation2.2.6_1}) for different $\alpha$ and $h_1$, with $a=\epsilon=1$, $b=4$, $h_2=h_1/2$, and $u(x)=x^{1+\alpha}$.}\label{table2.2.1_1}
\begin{tabular}{ccccccccc}
\hline\noalign{\smallskip}
$\alpha$   &\multicolumn{2}{c}{$1.2$} &\multicolumn{2}{c}{$1.4$} &\multicolumn{2}{c}{$1.6$} &\multicolumn{2}{c}{$1.8$}\\
\noalign{\smallskip}\hline\noalign{\smallskip}
$h_1$& $e$& rate & $e$ & rate & $e$ & rate & $e$ & rate\\
\noalign{\smallskip}\hline\noalign{\smallskip}
1/10 &3.04 1e-1 &-    &3.42 1e-1 &-    &3.81 1e-1 &-    &4.08 1e-1 & -   \\
1/20 &1.52 1e-1 &1.00 &1.70 1e-1 &1.01 &1.88 1e-1 &1.02 &2.00 1e-1 & 1.03\\
1/40 &7.58 1e-2 &1.00 &8.47 1e-2 &1.00 &9.20 1e-2 &1.03 &8.62 1e-2 & 1.21\\
1/80 &3.79 1e-2 &1.00 &4.23 1e-2 &1.00 &4.58 1e-2 &1.01 &4.10 1e-2 & 1.07\\
1/160&1.89 1e-2 &1.00 &2.12 1e-2 &1.00 &2.29 1e-2 &1.00 &2.02 1e-2 & 1.02 \\
\noalign{\smallskip}\hline
\end{tabular}
\end{table}

As a matter of fact, since $u_1(x)$ and $u_2(x)$ are mutually independent, it is certainly possible to approximate them using different schemes, respectively. So, similar to the h-p finite element methods, in our provided finite difference schemes, we can flexibly choose small/big stepsize and low/high order scheme with the change of the regularities of the solution in different parts of the domain.


In the next subsection, we give two kinds of non-uniform meshes and some notations.

\subsection{Interval partitions and some notations}\label{subsec:2.2}

Now we list two kinds of non-uniform meshes; and the solution domain is $[0,b]$, and $0<a<a+2\epsilon<b$.
\begin{enumerate}
\item {Case 1: $h_1 \leq h_2$}

We partition the interval $[0,a]$ into a uniform mesh with the stepsize $h_1=a/(N_1+1)$:
\begin{equation}\label{equation2.3.1}
0=x_0<x_1<\cdots<x_{N_1}<x_{N_1+1}=a,
\end{equation}
where $x_i=ih_1$, $0\leq i\leq [\frac{b}{h_1}]$, and $x_i\in [a,b]$ when $i>N_1+1$; and supposing that $a+2\epsilon$ can be divisible by $h_2:=m h_1$, $m\in \mathbb{N}_{+}$,
we partition the interval $[0,b]$ into a uniform mesh with the stepsize $h_2$:
\begin{equation}\label{equation2.3.2}
0=y_{-\frac{N_1+N_I}{m}}<\cdots<a+2\epsilon=y_{0}<y_1<\cdots<y_{N_2}<y_{N_2+1}=b,
\end{equation}
where $y_j=x_{N_1+N_I+jm}$, $-\frac{N_1+N_I}{m}\leq j\leq N_2+1$; and we partition the interval $[a,a+2\epsilon]$ into a uniform mesh with the stepsize $h_1$:
\begin{equation}\label{equation2.3.3}
a=z_{1}<z_2<\cdots<z_{N_I}=a+2\epsilon,
\end{equation}
where $z_k=x_{N_1+k}$, $1\leq k \leq N_I$.

\item {Case 2: $h_2 \leq h_1$}

We partition the interval $[a+2\epsilon,b]$ into a uniform mesh with the stepsize $h_2=(b-a-2\epsilon)/(N_2+1)$:
\begin{equation}\label{equation2.3.4}
a+2\epsilon=y_0<y_1<\cdots<y_{N_2}<y_{N_2+1}=b,
\end{equation}
where $y_i=ih_2$, $-(N_1+1)m-N_I+1\leq i\leq N_2+1$, and $y_i \in [0,a+2\epsilon]$ when $i<0$; and supposing that $b-a$ can be divisible by $h_1:=m h_2$, $m\in \mathbb{N}_{+}$,
we partition the interval $[0,b]$ into a uniform mesh with the stepsize $h_1$:
\begin{equation}\label{equation2.3.5}
0=x_{0}<x_1<\cdots<x_{N_1}<x_{N_1+1}=a<\cdots<x_{\frac{N_2+N_I}{m}+N_1+1}=b,
\end{equation}
where $x_j=y_{(j-N_1-1)m-N_I+1}$, $0\leq j\leq N_1+1+\frac{N_I+N_2}{m}$; and we partition the interval $[a,a+2\epsilon]$ into a uniform mesh with the stepsize $h_2$:
\begin{equation}\label{equation2.3.6}
a=z_{1}<z_2<\cdots<z_{N_I}=a+2\epsilon,
\end{equation}
where $z_k=y_{k-N_I}$, $1\leq k \leq N_I$.
\end{enumerate}

Denoting $N:=N_1+N_I+N_2+1$, the above grids on the interval $[0,b]$ can be re-labeled on non-uniform meshes as
\begin{eqnarray}\label{equation2.3.7}
&0=\tilde{x}_{0}<\tilde{x}_{1}<\cdots<\tilde{x}_{N_1}<\tilde{x}_{N_1+1}=
a<\cdots<\tilde{x}_{N_1+N_I}
\nonumber\\
&=a+2\epsilon<\tilde{x}_{N_1+N_I+1}<\cdots<\tilde{x}_{N_1+N_I+N_2}<\tilde{x}_{N}=b,
\end{eqnarray}
where
\begin{equation}\label{equation2.3.8}
\tilde{x}_n=\left\{ \begin{array}{lll}
x_n,&~~0\leq n \leq N_1+N_I,\\
y_{n-N_1-N_I},&~~ N_1+N_I+1\leq n \leq N,
\end{array} \right.
\end{equation}
if $h_1\leq h_2$; and
\begin{equation}\label{equation2.3.9}
\tilde{x}_n=\left\{ \begin{array}{lll}
x_n,&~~0\leq n \leq N_1,\\
y_{n-N_1-N_I},&~~ N_1+1\leq n \leq N,
\end{array} \right.
\end{equation}
if $h_2 \leq h_1$.

From Figures \ref{fig:2.3.1} and \ref{fig:2.3.2}, we can have a better understanding of the relations among these partitions with the truncated functions.

\begin{figure}[!htbp]
\includegraphics[scale=0.28]{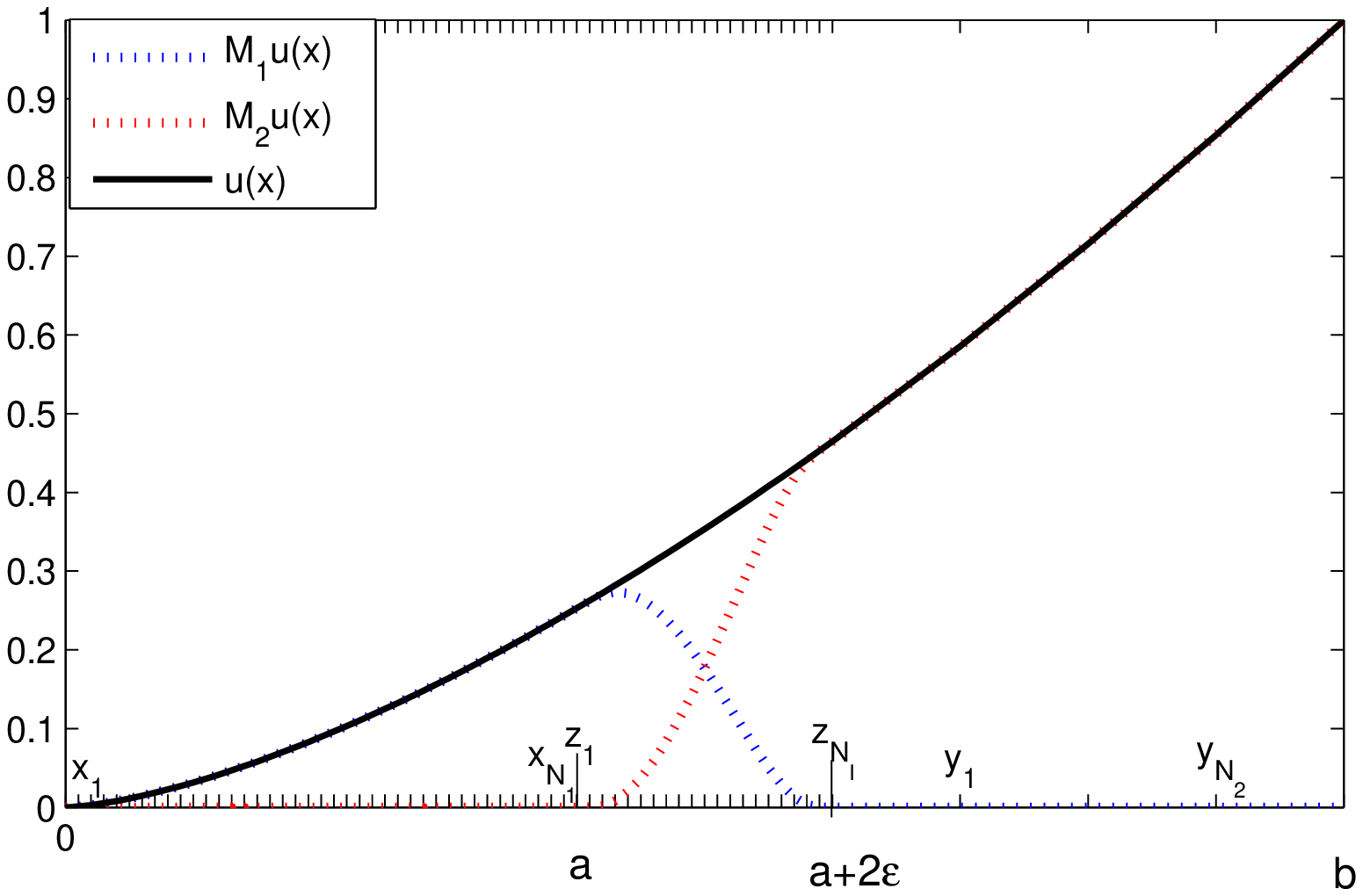}\\
\caption{An illustration of the partitions and the smooth truncations for $u(x)=x^{1.5}$, where $a=0.4$, $b=1$, $\epsilon=0.1$, and $h_1=1/100$, $h_2=1/10$.}\label{fig:2.3.1}
\includegraphics[scale=0.3]{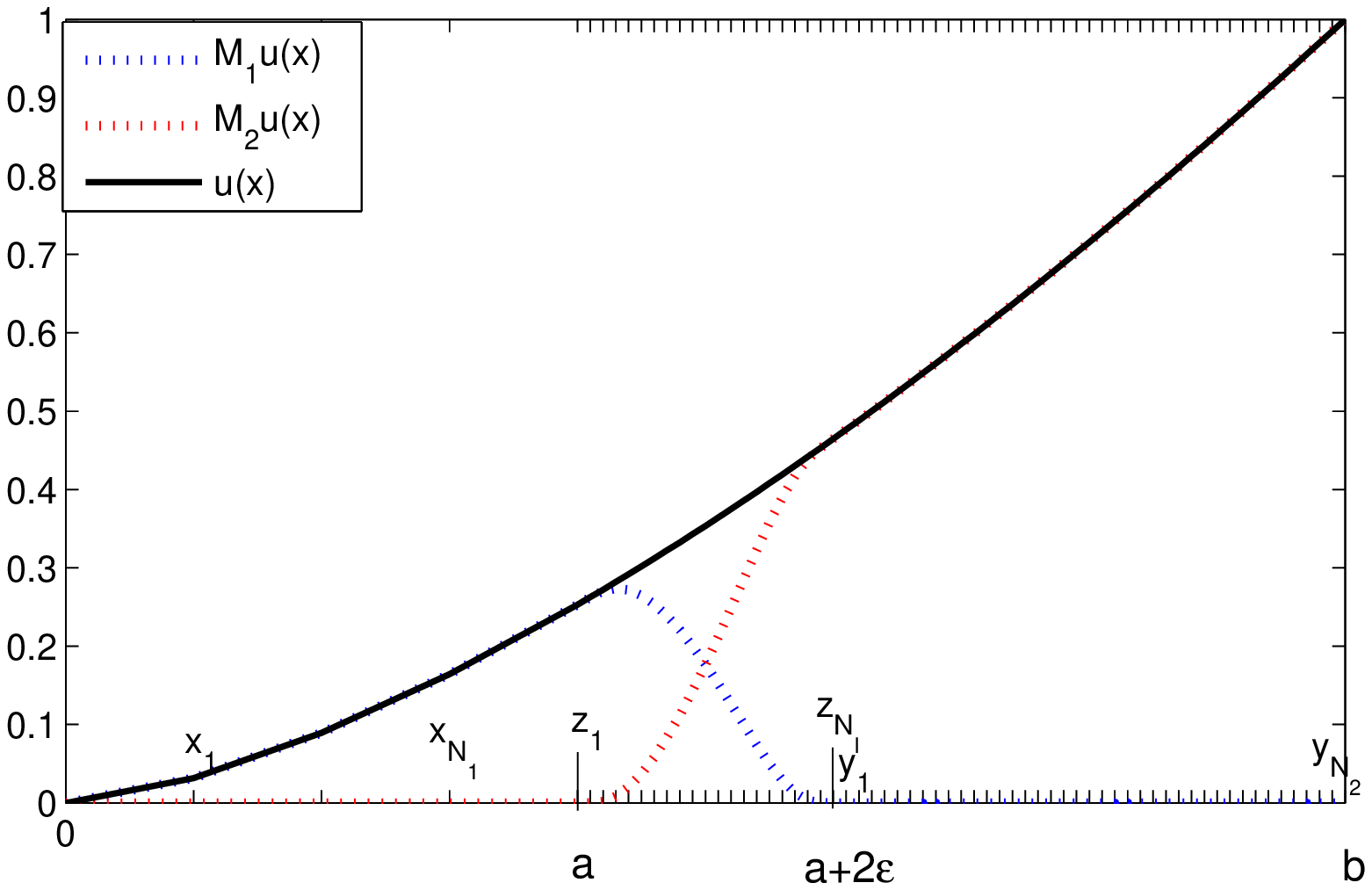}\\
\caption{An illustration of the partitions and the smooth truncations for $u(x)=x^{1.5}$, where $a=0.4$, $b=1$, $\epsilon=0.1$, and $h_2=1/100$, $h_1=1/10$.}\label{fig:2.3.2}
\end{figure}

\subsection{Approximation to the fractional derivative on non-uniform meshes}\label{subsec:2.3}

In this subsection, we derive several approximations to the fractional derivative on non-uniform meshes.

By Theorem $2$ in \cite{Zhao:14}, if $f(x)\in C^{l}[a,b]$, $D^{l+1}u(x)$ $ \in L^1(a,b)$, and $D^{k}u(a)=0,~k=0,1,\cdots,l-1$, then there are two sets of data $\{c_{-1}, c_0, c_1\}$ and $\{d_{-1},d_0,d_1\}$, such that
\begin{eqnarray}\label{equation2.4.1}
&&c_{-1}\,_{0}D_{x}^{\alpha}f(x-h)+c_0\,_{0}D_{x}^{\alpha}f(x)+c_1\,_{0}D_{x}^{\alpha}f(x+h)
\nonumber\\
&=&
d_{-1}\delta_{h,-1}^{\alpha}f(x)+d_{0}\delta_{h,0}^{\alpha}f(x)
+d_{1}\delta_{h,1}^{\alpha}f(x)+O(h^{l})
\end{eqnarray}
holds for any given $l\in \mathbb{N}_+$, where
\begin{equation}\label{equation2.4.2}
\delta^{\alpha}_{h,p}f(x):= \frac{1}{h^\alpha}\sum_{k=0}^{[\frac{x}{h}]+p}g_k^{(\alpha)}f(x-(k-p)h),~~~p\in \mathbb{Z}.
\end{equation}

For the convenience of presentation, we just discuss some selected schemes of (\ref{equation2.4.1}) on non-uniform meshes, where the parameters
$\left\{c_{-1},c_0,c_1\right\}$ are given as $\left\{0,1,0\right\}$ and $l=1,2$,
such as the first order shifted Gr\"{u}nwald  formula, and a family of the second order approximations discussed in \cite{Li:13}.
The process is similar for the other methods.

Firstly, for Case 1 in Subsection \ref{subsec:2.2} where $h_1\leq h_2$, we suppose that function $u(x):[0,b]\rightarrow \mathbb{R}$ for some $b>0$ satisfies: $u(x)\in \mathbf{C}^1[0,b]$, $D^2 u(x)\in L^1(0,b)$, and $D^l u (x)\in \mathbf{C}[\delta,b]$ for any $\delta>0$ and $l>1$. That is, $u(x)$ is sufficiently smooth except at the left endpoint, and $\,_{0}D_{x}^{\alpha}u(x)$, $1<\alpha\leq 2$, can still be approximated with the first order accuracy by combining the shifted Gr\"{u}nwald difference operator
with the techniques introduced in \cite{Chen:14} and \cite{Zhao:14}.
Specifically,  the approximation is given as
\begin{eqnarray}\label{equation2.4.3}
&&\,_{0}D_{x}^{\alpha}u(x)
\nonumber\\
&=&\,_{0}D_{x}^{\alpha}\big[u(x)-u_0+u_0\big]
\nonumber\\
&:=&\,_{0}D_{x}^{\alpha} r(x)+\,_{0}D_{x}^{\alpha}u_0
\nonumber\\
&=&h^{-\alpha}\sum_{k=0}^{[\frac{x}{h}]+p}g_k^{(\alpha)} r\big(x-(k-p)h\big)
+u_0\frac{x^{-\alpha}}{\Gamma(1-\alpha)}+O(h)
\nonumber\\
&=&h^{-\alpha}\sum_{k=0}^{[\frac{x}{h}]+p}g_k^{(\alpha)} \Big[u\big(x-(k-p)h\big)-u_0\Big]
+u_0\frac{x^{-\alpha}}{\Gamma(1-\alpha)}+O(h)
\nonumber
\end{eqnarray}
\begin{eqnarray}
&=&h^{-\alpha}\bigg[\sum_{k=0}^{[\frac{x}{h}]+p}g_k^{(\alpha)} u\big(x-(k-p)h\big)-\sum_{k=0}^{[\frac{x}{h}]+p}g_k^{(\alpha)} u_0\bigg]+u_0\frac{x^{-\alpha}}{\Gamma(1-\alpha)}
\nonumber\\
&&+O(h)
\nonumber\\
&=&h^{-\alpha}\sum_{k=0}^{[\frac{x}{h}]+p}g_k^{(\alpha)} u\big(x-(k-p)h\big)+u_0\bigg[-h^{-\alpha}\sum_{k=0}^{[\frac{x}{h}]+p}g_k^{(\alpha)}
+\frac{x^{-\alpha}}{\Gamma(1-\alpha)}\bigg]
\nonumber\\
&&+O(h),
\end{eqnarray}
where $u_0:=u(0)$, $r(x)=u(x)-u_0$, $r(x)\in \mathbf{C}^1[0,b]$, $D^2 r(x)\in L^1(0,b)$,
also $r(0)=0$.


According to Lemma \ref{lemma:2.2.1}, there exists  $\,_{0}D_{x}^{\alpha}r(x)=\,_{0}D_{x}^{\alpha}r_1(x)+\,_{0}D_{x}^{\alpha}r_2(x)$,
where $r_j(x)=\mathcal{M}_j^{a,\epsilon}r(x)$, $j=1,2$.
Hence, computing $\,_{0}D_{x}^{\alpha}r(x)$ is equivalent to computing $\,_{0}D_{x}^{\alpha}r_1(x)$ and $\,_{0}D_{x}^{\alpha}r_2(x)$. By the regularity assumption of $u(x)$ and Lemma \ref{lemma:2.2.1}, for ensuring accuracy with small computational cost, $\,_{0}D_{x}^{\alpha}r_1(x)$ can be approximated by the first order method (\ref{equation2.4.3}) with fine scale, while $\,_{0}D_{x}^{\alpha}r_2(x)$ by high order methods with coarse scale.


Combining with the non-uniform partitions and notations in Case 1 of Subsection \ref{subsec:2.2}, there are
\begin{eqnarray}\label{equation2.4.4}
&&\,_{0}D_{x}^{\alpha}r_1(x_n)
\nonumber\\
&=&h_1^{-\alpha}\sum_{k=0}^{n}g_k^{(\alpha)} r_{1}(x_{n-k+1})
+O(h_1)
\nonumber\\
&=&h_1^{-\alpha}\sum_{k=0}^{n}g_k^{(\alpha)} u_{1}(x_{n-k+1})
-h_1^{-\alpha}\sum_{k=0}^{n}g_k^{(\alpha)} u_{0,1}(x_{n-k+1})
+O(h_1),
\end{eqnarray}
for $n=1,2,\cdots,N_1+N_I,N_1+N_I+m,\cdots,N_1+N_I+m N_2$, where $x_n=nh_1$, $u_{1}(x)=\mathcal{M}_1^{a,\epsilon}u(x)$,
$u_{0,1}(x)=\mathcal{M}_1^{a,\epsilon}u_0(x)$;

\begin{eqnarray}\label{equation2.4.5}
&&\,_{0}D_{x}^{\alpha}r_2(x_{N_1+N_I+nm})
\nonumber\\
&=&\,_{0}D_{x}^{\alpha}r_2(y_n)
\nonumber\\
&=&h_2^{-\alpha}\sum_{k=0}^{n}w_k^{(2,\alpha)} r_{2}(y_{n-k+1})
+O(h_2^{l_2})
\nonumber\\
&=&h_2^{-\alpha}\sum_{k=0}^{n}w_k^{(2,\alpha)} u_{2}(y_{n-k+1})
-h_2^{-\alpha}\sum_{k=0}^{n}w_k^{(2,\alpha)} u_{0,2}(y_{n-k+1})
+O(h_2^{l_2})
\nonumber\\
&=&h_2^{-\alpha}\sum_{k=0}^{n}w_k^{(2,\alpha)} u_{2}(x_{N_1+N_I+(n-k+1)m})
\nonumber\\
&&-h_2^{-\alpha}\sum_{k=0}^{n}w_k^{(2,\alpha)} u_{0,2}(x_{N_1+N_I+(n-k+1)m})
+O(h_2^{l_2}),
\end{eqnarray}
for $n=-\frac{N_1+N_I}{m}+1,\cdots,N_2$, where
$u_{2}(x)=\mathcal{M}_2^{a,\epsilon}u(x)$,
$u_{0,2}(x)=\mathcal{M}_2^{a,\epsilon}u_0(x)$, and
\begin{equation}\label{equation2.4.6}
\left\{ \begin{array}{lll}
w_0^{(2,\alpha)}&=&d_1^{(2)} g_0^{(\alpha)},\\
w_1^{(2,\alpha)}&=&d_0^{(2)} g_0^{(\alpha)}+d_1^{(2)} g_1^{(\alpha)},\\
w_k^{(2,\alpha)}&=&d_{-1}^{(2)} g_{k-2}^{(\alpha)}+d_0^{(2)} g_{k-1}^{(\alpha)}+d_1^{(2)} g_k^{(\alpha)},~k\geq 2,
\end{array} \right.
\end{equation}
and the value of $l_2$ depends on the choosing of $d_{-1}^{(2)}$, $d_{0}^{(2)}$, and $d_{1}^{(2)}$ \cite{Zhao:14}, i.e., for sufficiently smooth function $v(x)$,
\begin{eqnarray}
\,_{0}D_{x}^{\alpha}v(x)=
d_{-1}^{(2)}\delta_{h,-1}^{\alpha}v(x)+d_{0}^{(2)}\delta_{h,0}^{\alpha}v(x)
+d_{1}^{(2)}\delta_{h,1}^{\alpha}v(x)+O(h^{l_2}).
\end{eqnarray}

While for $\,_{0}D_{x}^{\alpha}r_2(x_{N_1+N_I+nm+q})$, $n=-\frac{N_1+N_I}{m}+1,\cdots,-1$, and $q=1,2,\cdots,m-1$, since $r_2(x)$ is smooth enough, by the linear interpolation there exists
\begin{eqnarray}\label{equation2.4.7}
&&\,_{0}D_{x}^{\alpha}r_2(x_{N_1+N_I+nm+q})
\nonumber\\
&=&(1-\frac{q}{m})\,_{0}D_{x}^{\alpha}r_2(y_n)
+\frac{q}{m}\,_{0}D_{x}^{\alpha}r_2(y_{n+1})
+O(h_2^{l_2})
\nonumber\\
&=&(1-\frac{q}{m})h_2^{-\alpha}\sum_{k=0}^{n}w_k^{(2,\alpha)} r_{2}(y_{n-k+1})
+\frac{q}{m}h_2^{-\alpha}\sum_{k=0}^{n+1}w_k^{(2,\alpha)} r_{2}(y_{n-k+2})
+O(h_2^{l_2})
\nonumber\\
&=&h_2^{-\alpha}\sum_{k=0}^{n+1}\varpi_{k,q}^{(2,\alpha)} u_{2}(y_{n-k+2})
-h_2^{-\alpha}\sum_{k=0}^{n+1}\varpi_{k,q}^{(2,\alpha)} u_{0,2}(y_{n-k+2})
+O(h_2^{l_2})
\nonumber\\
&=&h_2^{-\alpha}\sum_{k=0}^{n+1}\varpi_{k,q}^{(2,\alpha)} u_{2}(x_{N_1+N_I+(n-k+2)m})
\nonumber\\
&&-h_2^{-\alpha}\sum_{k=0}^{n+1}\varpi_{k,q}^{(2,\alpha)} u_{0,2}(x_{N_1+N_I+(n-k+2)m})
+O(h_2^{l_2}),
\end{eqnarray}
where
\begin{equation}\label{equation2.4.8}
\left\{ \begin{array}{lll}
\varpi_{0,q}^{(2,\alpha)}&=& \frac{q}{m}w_{0}^{(2,\alpha)},\\
\varpi_{k,q}^{(2,\alpha)}&=& (1-\frac{q}{m})w_{k-1}^{(2,\alpha)}+ \frac{q}{m}w_{k}^{(2,\alpha)},~k\geq 1.
\end{array} \right.
\end{equation}
Note that there are
\begin{equation}\label{equation2.4.8.add}
\varpi_{k+1,0}^{(2,\alpha)}=w_{k}^{(2,\alpha)},~~~
\varpi_{k,m}^{(2,\alpha)}=w_{k}^{(2,\alpha)}.
\end{equation}

Combining Eq. (\ref{equation2.4.4})-(\ref{equation2.4.7}) with
\begin{eqnarray}\label{equation2.4.9}
\,_{0}D_{x}^{\alpha}u(\tilde{x}_n)
=\,_{0}D_{x}^{\alpha} r_1(\tilde{x}_n)+\,_{0}D_{x}^{\alpha} r_2(\tilde{x}_n)
+u_0\frac{\tilde{x}_n^{-\alpha}}{\Gamma(1-\alpha)},
\end{eqnarray}
for $n=1,2,\cdots,N-1$, we can get
\begin{eqnarray}\label{equation2.4.10}
&&\,_{0}D_{x}^{\alpha}u(\tilde{x}_{mn})
\nonumber\\
&=&h_1^{-\alpha}\sum_{k=0}^{mn}g_k^{(\alpha)} u_{1}(\tilde{x}_{mn-k+1})
+h_2^{-\alpha}\sum_{k=0}^{n}w_k^{(2,\alpha)} u_{2}(\tilde{x}_{(n-k+1)m})
\nonumber\\
&&-h_1^{-\alpha}\sum_{k=0}^{mn}g_k^{(\alpha)} u_{0,1}(\tilde{x}_{mn-k+1})
-h_2^{-\alpha}\sum_{k=0}^{n}w_k^{(2,\alpha)} u_{0,2}(\tilde{x}_{(n-k+1)m})
\nonumber\\
&&+h_1^{-\alpha}\frac{(mn)^{-\alpha}}{\Gamma(1-\alpha)}u_0+O(h_1)+O(h_2^{l_2}),
\end{eqnarray}
for $n=1,2,\cdots,\frac{N_1+N_I}{m}$;
\begin{eqnarray}\label{equation2.4.11}
&&\,_{0}D_{x}^{\alpha}u(\tilde{x}_{mn+q})
\nonumber\\
&=&h_1^{-\alpha}\sum_{k=0}^{mn+q}g_k^{(\alpha)} u_{1}(\tilde{x}_{mn+q-k+1})
+h_2^{-\alpha}\sum_{k=0}^{n+1}\varpi_{k,q}^{(2,\alpha)} u_{2}(\tilde{x}_{(n-k+2)m})
\nonumber\\
&&-h_1^{-\alpha}\sum_{k=0}^{mn+q}g_k^{(\alpha)} u_{0,1}(\tilde{x}_{mn+q-k+1})
-h_2^{-\alpha}\sum_{k=0}^{n+1}\varpi_{k,q}^{(2,\alpha)} u_{0,2}(\tilde{x}_{(n-k+2)m})
\nonumber\\
&&+h_1^{-\alpha}\frac{(mn+q)^{-\alpha}}{\Gamma(1-\alpha)}u_0+O(h_1)+O(h_2^{l_2}),
\end{eqnarray}
for $q=1,2,\cdots,m-1$ and $n=0,1,\cdots,\frac{N_1+N_I}{m}-1$; and
\begin{eqnarray}\label{equation2.4.12}
&&\,_{0}D_{x}^{\alpha}u(\tilde{x}_{n})
\nonumber\\
&=&h_1^{-\alpha}\bigg[\sum_{k=n+2-N_1-N_I}^{n}g_k^{(\alpha)} u_{1}(\tilde{x}_{n-k+1})
+m^{-\alpha}\sum_{k=0}^{n+1-N_1-N_I} w_k^{(2,\alpha)} u_2(\tilde{x}_{n-k+1})
\nonumber\\
&&~~~~~~~~+m^{-\alpha}\sum_{k=n+2-N_1-N_I}^{n-\frac{m-1}{m}(N_1+N_I)} w_k^{(2,\alpha)} u_2(\tilde{x}_{N_1+N_I+(n-N_1-N_I-k+1)m})\bigg]
\nonumber\\
&&-h_1^{-\alpha}\bigg[\sum_{k=n+2-N_1-N_I}^{n}g_k^{(\alpha)} u_{0,1}(\tilde{x}_{n-k+1})
\nonumber\\
&&~~~~~~~~+m^{-\alpha}\sum_{k=0}^{n+1-N_1-N_I} w_k^{(2,\alpha)} u_{0,2}(\tilde{x}_{n-k+1})
\nonumber\\
&&~~~~~~~~+m^{-\alpha}\sum_{k=n+2-N_1-N_I}^{n-\frac{m-1}{m}(N_1+N_I)} w_k^{(2,\alpha)} u_{0,2}(\tilde{x}_{N_1+N_I+(n-N_1-N_I-k+1)m})\bigg]
\nonumber\\
&&+h_1^{-\alpha}
\frac{\left[(n-N_1-N_I)m+N_1+N_I\right]^{-\alpha}}{\Gamma(1-\alpha)}u_0+O(h_1)+O(h_2^{l_2}),
\end{eqnarray}
for $n=N_1+N_I+1,\cdots,N_1+N_I+N_2$.

\begin{remark}\label{remark:2.4.1}
From the above equalities (\ref{equation2.4.10})-(\ref{equation2.4.12}), it can be noticed that if $u(x)$ is taken just as a constant $u_0$, then $u_1=u_{0,1}$, $u_2=u_{0,2}$, and the approximations are exact. In other words, for a nonhomogeneous function $u(x)$, the approximation accuracy is the same as its corresponding homogeneous one $u(x)-u_0$. That is why in the above methods, we divide the function $r(x)=u(x)-u_0$ into $r_1(x)$ and $r_2(x)$, rather than directly divide $u(x)$ into $u_1(x)$ and $u_2(x)$.
\end{remark}

\begin{remark}\label{remark:2.4.2}
Noting that the number of terms in the summation related to $u_1(x)$ in (\ref{equation2.4.12}) equals to $N_1+N_I-1$ and does not change with $n$, we can see the cost of calculation in this way can be reduced notably than computing $\,_{0}D_{x}^{\alpha}u(x)$ on $(0,b)$ totally by the first order method (\ref{equation2.4.3}) with fine stepsize $h_1$.
\end{remark}

As for Case 2 in Subsection \ref{subsec:2.2}, we simply assume that $u(x)\in \mathbf{C}^2[0,b]$, $D^3 u(x)\in L^1(0,b)$, and $u(0)=u'(0)=0$. At this time, we have
\begin{eqnarray}\label{equation2.4.13}
&&\,_{0}D_{x}^{\alpha}u_2(y_n)
\nonumber\\
&=&h_2^{-\alpha}\sum_{k=0}^{n+N_1m+N_I}w_k^{(2,\alpha)} u_{2}(y_{n-k+1})
+O(h_2^{l_2}),
\end{eqnarray}
for $n=-N_1m-N_I,\cdots,N_2$, where $y_n=nh_2$, $u_{2}(y)=\mathcal{M}_2^{a,\epsilon}u(y)$;

\begin{eqnarray}\label{equation2.4.14}
&&\,_{0}D_{x}^{\alpha}u_1(y_{(n-N_1-1)m-N_I+1})
\nonumber\\
&=&\,_{0}D_{x}^{\alpha}u_1(x_{n})
\nonumber\\
&=&h_1^{-\alpha}\sum_{k=0}^{n}w_k^{(1,\alpha)} u_{1}(x_{n-k+1})
+O(h_1^{l_1})
\nonumber\\
&=&h_1^{-\alpha}\sum_{k=0}^{n}w_k^{(1,\alpha)} u_{1}(y_{(n-N_1-k)m-N_I+1})+O(h_1^{l_1}),
\end{eqnarray}
for $n=1,\cdots,N_1+1+\frac{N_2+N_I}{m}$, where
$u_{1}(y)=\mathcal{M}_1^{a,\epsilon}u(y)$, and $w_k^{(j,\alpha)},~j=1,2$, are defined as in (\ref{equation2.4.6});
while for $\,_{0}D_{x}^{\alpha}u_1(y_{(n-N_1-1)m-N_I+1+q})$, $n=N_1+1,\cdots,N_1+\frac{N_2+N_I}{m}$, and $q=1,2,\cdots,m-1$, since $u_1(x)$ is smooth enough, by the linear interpolation there exists
\begin{eqnarray}\label{equation2.4.15}
&&\,_{0}D_{x}^{\alpha}u_1(y_{(n-N_1-1)m-N_I+1+q})
\nonumber\\
&=&\left(1-\frac{q}{m}\right)\,_{0}D_{x}^{\alpha}u_1(x_n)
+\frac{q}{m}\,_{0}D_{x}^{\alpha}u_1(x_{n+1})
+O(h_1^2)
\nonumber\\
&=&\left(1-\frac{q}{m}\right)h_1^{-\alpha}\sum_{k=0}^{n}w_k^{(1,\alpha)} u_{1}(x_{n-k+1})
+\frac{q}{m}h_1^{-\alpha}\sum_{k=0}^{n+1}w_k^{(1,\alpha)} u_{1}(x_{n-k+2})
+O(h_1^{l_1})
\nonumber\\
&=&h_1^{-\alpha}\sum_{k=0}^{n+1}\varpi_{k,q}^{(1,\alpha)} u_{1}(x_{n-k+2})
+O(h_1^{l_1})
\nonumber\\
&=&h_1^{-\alpha}\sum_{k=0}^{n+1}\varpi_{k,q}^{(1,\alpha)} u_{1}(y_{(n-N_1-k+1)m-N_I+1})
+O(h_1^{l_1}),
\end{eqnarray}
where $\varpi_{k,q}^{(1,\alpha)}$ are similarly defined as the ones in (\ref{equation2.4.8}).

Combining Eq. (\ref{equation2.4.13})-(\ref{equation2.4.15}) with
\begin{eqnarray}\label{equation2.4.16}
\,_{0}D_{x}^{\alpha}u(\tilde{x}_n)
=\,_{0}D_{x}^{\alpha} u_1(\tilde{x}_n)+\,_{0}D_{x}^{\alpha} u_2(\tilde{x}_n),
\end{eqnarray}
for $n=1,2,\cdots,N-1$, we get
\begin{eqnarray}\label{equation2.4.17}
&&\,_{0}D_{x}^{\alpha}u(\tilde{x}_{n})
\nonumber\\
&=&h_1^{-\alpha}\sum_{k=0}^{n}w_k^{(1,\alpha)} u_{1}(\tilde{x}_{n-k+1})
+h_2^{-\alpha}\sum_{k=0}^{mn}w_k^{(2,\alpha)} u_{2}(\tilde{x}_{mn-k+1})
\nonumber\\
&&+O(h_1^{l_1})+O(h_2^{l_2}),
\end{eqnarray}
for $n=1,2,\cdots,N_1+1$;

\begin{eqnarray}\label{equation2.4.18}
&&\,_{0}D_{x}^{\alpha}u(\tilde{x}_{(n-N_1-1)m+N_1+1})
\nonumber\\
&=&h_2^{-\alpha}\bigg[m^{-\alpha}\sum_{k=0}^{n-N_1}w_k^{(1,\alpha)} u_{1}(\tilde{x}_{(n-k-N_1)m+N_1+1})
\nonumber\\
&&~~~~~~+m^{-\alpha}\sum_{k=n-N_1+1}^{n} w_k^{(1,\alpha)} u_{1}(\tilde{x}_{n-k+1})
\nonumber\\
&&~~~~~~
+\sum_{k=0}^{(n-N_1-1)m+1} w_k^{(2,\alpha)} u_2(\tilde{x}_{(n-N_1-1)m+N_1+2-k})\bigg]
\nonumber\\
&&+O(h_1^{l_1})+O(h_2^{l_2}),
\end{eqnarray}
for $n=N_1+1,N_1+1+m,\cdots,N_1+1+\frac{N_I+N_2}{m}$;

\begin{eqnarray}\label{equation2.4.19}
&&\,_{0}D_{x}^{\alpha}u(\tilde{x}_{(n-N_1-1)m+N_1+1+q})
\nonumber\\
&=&h_2^{-\alpha}\bigg[m^{-\alpha}\sum_{k=0}^{n-N_1+1}\varpi_{k,q}^{(1,\alpha)} u_{1}(\tilde{x}_{(n-k-N_1+1)m+N_1+1})
\nonumber\\
&&~~~~~~+m^{-\alpha}\sum_{k=n-N_1+2}^{n+1} \varpi_{k,q}^{(1,\alpha)} u_{1}(\tilde{x}_{n-k+2})
\nonumber\\
&&~~~~~~
+\sum_{k=0}^{(n-N_1-1)m+1+q} w_k^{(2,\alpha)} u_2(\tilde{x}_{(n-N_1-1)m+N_1+2-k+q})\bigg]
\nonumber\\
&&+O(h_1^{l_1})+O(h_2^{l_2}),
\end{eqnarray}
for $q=1,2,\cdots,m-1$, and $n=N_1+1,N_1+1+m,\cdots,N_1+1+\frac{N_I+N_2}{m}$.

\begin{remark}\label{remark:2.4.3}
For the case $h_2\leq h_1$, there is no difficulty to generate the above non-uniform schemes for the function $u(x)$ with $u(0)\neq0$ or $u'(0)\neq0$, by using the techniques introduced in \cite{Chen:14} and \cite{Zhao:14}. We omit the details here.
\end{remark}

\section{Application to Space Fractional Diffusion Equations}\label{sec:3}
In this section, we develop a general non-uniform scheme for space fractional diffusion equation:
\begin{equation}\label{equation3.0.1}
\left\{ \begin{array}{lll}
\frac{\partial u(x,t) }{\partial t}&=&K\,_{0}D_x^{\alpha}u(x,t)+f(x,t) \\
&& ~~~~ ~~~~ ~~~ {\rm for}~~~ (x,t) \in (0,b)\times (0,T),\\
u(x,0) &=&\phi_0(x) ~~~~ {\rm for}~~~ x \in [0,b], \\
u(0,t)&=&u_{0}(t) ~~~~{\rm for}~~~ t\in[0,T],\\
u(b,t)&=&u_{b}(t) ~~~~{\rm for}~~~ t\in[0,T],
\end{array} \right.
\end{equation}
where $_{0}D_x^{\alpha}$ is the left Riemann-Liouville fractional
derivative with $1<\alpha \leq 2$. The diffusion
coefficient $K$ is a nonnegative constant.

\subsection{Numerical schemes on non-uniform meshes}\label{subsec:3.1}

We partition the interval $[0,b]$ as discussed in Subsection \ref{subsec:2.2} and denote the time steplength
$\tau=T/M$, where $M$ is a positive integer. The mesh points in time direction are denoted by $t_i=i\tau$ for $0\leq i \leq M$. Let $t_{i+1/2}=(t_i+t_{i+1})/2$ for $0\leq i \leq M-1$. And the
following notations are used in the sections below
\begin{eqnarray}\label{equation3.0.2}
&&u_n^i=u(\tilde{x}_n,t_i),~~u_{j,n}^i=u_j(\tilde{x}_n,t_i),~j=1,2,
\nonumber\\
&&f_n^{i+1/2}=f(\tilde{x}_n,t_{i+1/2}),
\nonumber\\
&&u_0^{i}=u_0(t_i),~~u_{0,j,n}^{i}=u_{0,j}(\tilde{x}_n,t_i),~j=1,2.
\end{eqnarray}

Let $u(x,\cdot)\in \mathbf{C}^1[0,b]$, $D^2 u(x,\cdot)\in L^1(0,b)$, and $D^2 u (x,\cdot)\in \mathbf{C}[\delta,b]$ for any $\delta>0$, and $h_1\leq h_2$. Using the Crank-Nicolson (CN) technique to discretize the time derivative of (\ref{equation3.0.1}) and non-uniform discretization (\ref{equation2.4.10})-(\ref{equation2.4.12}) in space direction leads to
\begin{eqnarray}\label{equation3.0.3}
&&u_{mn}^{i+1}
-\frac{\tau}{2}K h_1^{-\alpha}\bigg[
\sum_{k=0}^{mn}g_k^{(\alpha)} u_{1,mn-k+1}^{i+1}
+m^{-\alpha}\sum_{k=0}^{n}w_k^{(2,\alpha)} u_{2,(n-k+1)m}^{i+1}\bigg]
\nonumber\\
&=&u_{mn}^{i}
+\frac{\tau}{2}K h_1^{-\alpha}\bigg[
\sum_{k=0}^{mn}g_k^{(\alpha)} u_{1,mn-k+1}^{i}
+m^{-\alpha}\sum_{k=0}^{n}w_k^{(2,\alpha)} u_{2,(n-k+1)m}^{i}\bigg]
\nonumber\\
&&-\frac{\tau}{2}K h_1^{-\alpha}\bigg[
\sum_{k=0}^{mn}g_k^{(\alpha)} \left(u_{0,1,mn-k+1}^{i}+u_{0,1,mn-k+1}^{i+1}\right)
\nonumber\\
&&~~~~~~~~~~~~~~+m^{-\alpha}\sum_{k=0}^{n}w_k^{(2,\alpha)}\left( u_{0,2,(n-k+1)m}^{i}+u_{0,2,(n-k+1)m}^{i+1}\right)
\bigg]
\nonumber\\
&&+\frac{\tau}{2}K h_1^{-\alpha}\frac{(mn)^{-\alpha}}{\Gamma(1-\alpha)}
\left(u_0^{i}+u_0^{i+1}\right)
\nonumber\\
&&+\tau f_{mn}^{i+1/2}
+O\left(\tau (h_1+h_2^{l_2})+\tau^3\right),
\end{eqnarray}
for $n=1,2,\cdots,\frac{N_1+N_I}{m}$;

\begin{eqnarray}\label{equation3.0.4}
&&u_{mn+q}^{i+1}
-\frac{\tau}{2}K h_1^{-\alpha}\bigg[
\sum_{k=0}^{mn+q}g_k^{(\alpha)} u_{1,mn+q-k+1}^{i+1}
+m^{-\alpha}\sum_{k=0}^{n+1}\varpi_{k,q}^{(2,\alpha)} u_{2,(n-k+2)m}^{i+1}\bigg]
\nonumber\\
&=&u_{mn+q}^{i}
+\frac{\tau}{2}K h_1^{-\alpha}\bigg[
\sum_{k=0}^{mn+q}g_k^{(\alpha)} u_{1,mn+q-k+1}^{i}
+m^{-\alpha}\sum_{k=0}^{n+1}\varpi_{k,q}^{(2,\alpha)} u_{2,(n-k+2)m}^{i}\bigg]
\nonumber\\
&&-\frac{\tau}{2}K h_1^{-\alpha}\bigg[
\sum_{k=0}^{mn+q}g_k^{(\alpha)} \left(u_{0,1,mn+q-k+1}^{i}+u_{0,1,mn+q-k+1}^{i+1}\right)
\nonumber\\
&&~~~~~~~~~~~~~~+m^{-\alpha}\sum_{k=0}^{n+1}\varpi_{k,q}^{(2,\alpha)}\left( u_{0,2,(n-k+2)m}^{i}+u_{0,2,(n-k+2)m}^{i+1}\right)
\bigg]
\nonumber\\
&&+\frac{\tau}{2}K h_1^{-\alpha}\frac{(mn+q)^{-\alpha}}{\Gamma(1-\alpha)}
\left(u_0^{i}+u_0^{i+1}\right)
\nonumber\\
&&+\tau f_{mn+q}^{i+1/2}
+O\left(\tau (h_1+h_2^{l_2})+\tau^3\right),
\end{eqnarray}
for $q=1,2,\cdots,m-1$ and $n=0,1,\cdots,\frac{N_1+N_I}{m}-1$; and

\begin{eqnarray}\label{equation3.0.5}
&&u_{n}^{i+1}
\nonumber\\
&&-\frac{\tau}{2}K h_1^{-\alpha}\bigg[
\sum_{k=n+2-N_1-N_I}^{n}g_k^{(\alpha)} u_{1,n-k+1}^{i+1}
+m^{-\alpha}\sum_{k=0}^{n+1-N_1-N_I} w_k^{(2,\alpha)} u_{2,n-k+1}^{i+1}
\nonumber\\
&&~~~~~~~~~~~~~~+m^{-\alpha}\sum_{k=n+2-N_1-N_I}^{n-\frac{m-1}{m}(N_1+N_I)} w_k^{(2,\alpha)} u_{2,N_1+N_I+(n-N_1-N_I-k+1)m}^{i+1}\bigg]
\nonumber\\
&=&u_{n}^{i}
\nonumber\\
&&+\frac{\tau}{2}K h_1^{-\alpha}\bigg[
\sum_{k=n+2-N_1-N_I}^{n}g_k^{(\alpha)} u_{1,n-k+1}^{i}
+m^{-\alpha}\sum_{k=0}^{n+1-N_1-N_I} w_k^{(2,\alpha)} u_{2,n-k+1}^{i}
\nonumber\\
&&~~~~~~~~~~~~~~+m^{-\alpha}\sum_{k=n+2-N_1-N_I}^{n-\frac{m-1}{m}(N_1+N_I)} w_k^{(2,\alpha)} u_{2,N_1+N_I+(n-N_1-N_I-k+1)m}^{i}\bigg]
\nonumber\\
&&-\frac{\tau}{2}K h_1^{-\alpha}\bigg[
\sum_{k=n+2-N_1-N_I}^{n}g_k^{(\alpha)} \left(u_{0,1,n-k+1}^{i}+u_{0,1,n-k+1}^{i+1}\right)
\nonumber\\
&&~~~~~~~~~~~~~~+m^{-\alpha}\sum_{k=0}^{n+1-N_1-N_I} w_k^{(2,\alpha)}\left( u_{0,2,n-k+1}^{i}+u_{0,2,n-k+1}^{i+1}\right)
\nonumber\\
&&~~~~~~~~~~~~~~+m^{-\alpha}\sum_{k=n+2-N_1-N_I}^{n-\frac{m-1}{m}(N_1+N_I)} w_k^{(2,\alpha)} \Big( u_{0,2,N_1+N_I+(n-N_1-N_I-k+1)m}^{i}
\nonumber\\
&&~~~~~~~~~~~~~~~~~~~~~~~~~~~~~~~~~~~~~~~~~~~~~~~~~~~~+u_{0,2,N_1+N_I+(n-N_1-N_I-k+1)m}^{i+1}\Big)
\bigg]
\nonumber
\end{eqnarray}
\begin{eqnarray}
&&+\frac{\tau}{2}K h_1^{-\alpha}\frac{\left[(n-N_1-N_I)m+N_1+N_I\right]^{-\alpha}}{\Gamma(1-\alpha)}
\left(u_0^{i}+u_0^{i+1}\right)
\nonumber\\
&&+\tau f_{n}^{i+1/2}
+O\left(\tau (h_1+h_2^{l_2})+\tau^3\right),
\end{eqnarray}
for $n=N_1+N_I+1,\cdots,N_1+N_I+N_2$.

Let $u(x,\cdot)\in \mathbf{C}^2[0,b]$, $D^3 u(x,\cdot)\in L^1(0,b)$, and $u(0,t)=u'(0,t)=0$, and $h_2\leq h_1$. Using the CN technique to discretize the time derivative of (\ref{equation3.0.1}) and non-uniform discretization (\ref{equation2.4.17})-(\ref{equation2.4.19}) in space direction leads to
\begin{eqnarray}\label{equation3.0.6}
&&u_{n}^{i+1}
-\frac{\tau}{2}K h_2^{-\alpha}\bigg[
m^{-\alpha}\sum_{k=0}^{n}w_k^{(1,\alpha)} u_{1,n-k+1}^{i+1}
+\sum_{k=0}^{mn}w_k^{(2,\alpha)} u_{2,mn-k+1}^{i+1}\bigg]
\nonumber\\
&=&u_{n}^{i}
+\frac{\tau}{2}K h_2^{-\alpha}\bigg[
m^{-\alpha}\sum_{k=0}^{n}w_k^{(1,\alpha)} u_{1,n-k+1}^{i}
+\sum_{k=0}^{mn}w_k^{(2,\alpha)} u_{2,mn-k+1}^{i}\bigg]
\nonumber\\
&&+\tau f_{n}^{i+1/2}+O\left(\tau (h_1^{l_1}+h_2^{l_2})+\tau^3\right),
\end{eqnarray}
for $n=1,2,\cdots,N_1+1$;

\begin{eqnarray}\label{equation3.0.7}
&&u_{(n-N_1-1)m+N_1+1}^{i+1}
\nonumber\\
&&-\frac{\tau}{2}K h_2^{-\alpha}\bigg[
m^{-\alpha}\sum_{k=0}^{n-N_1}w_k^{(1,\alpha)} u_{1,(n-k-N_1)m+N_1+1}^{i+1}
\nonumber\\
&&~~~~~~~~~~~~~~
+m^{-\alpha}\sum_{k=n-N_1+1}^{n} w_k^{(1,\alpha)}u_{1,n-k+1}^{i+1}
\nonumber\\
&&~~~~~~~~~~~~~~
+\sum_{k=0}^{(n-N_1-1)m+N_1+1} w_k^{(2,\alpha)} u_{2,(n-N_1-1)m+N_1+2-k}^{i+1}\bigg]
\nonumber\\
&=&u_{(n-N_1-1)m+N_1+1}^{i}
\nonumber\\
&&+\frac{\tau}{2}K h_2^{-\alpha}\bigg[
m^{-\alpha}\sum_{k=0}^{n-N_1}w_k^{(1,\alpha)} u_{1,(n-k-N_1)m+N_1+1}^{i}
\nonumber\\
&&~~~~~~~~~~~~~~
+m^{-\alpha}\sum_{k=n-N_1+1}^{n} w_k^{(1,\alpha)}u_{1,n-k+1}^{i}
\nonumber\\
&&~~~~~~~~~~~~~~
+\sum_{k=0}^{(n-N_1-1)m+N_1+1} w_k^{(2,\alpha)} u_{2,(n-N_1-1)m+N_1+2-k}^{i}\bigg]
\nonumber\\
&&+\tau f_{(n-N_1-1)m+N_1+1}^{i+1/2}+O\left(\tau (h_1^{l_1}+h_2^{l_2})+\tau^3\right),
\end{eqnarray}
for $n=N_1+1,N_1+1+m,\cdots,N_1+1+\frac{N_I+N_2}{m}$; and

\begin{eqnarray}\label{equation3.0.8}
&&u_{(n-N_1-1)m+N_1+1+q}^{i+1}
\nonumber\\
&&-\frac{\tau}{2}K h_2^{-\alpha}\bigg[
m^{-\alpha}\sum_{k=0}^{n-N_1+1}\varpi_{k,q}^{(1,\alpha)}
u_{1,(n-k-N_1+1)m+N_1+1}^{i+1}
\nonumber\\
&&~~~~~~~~~~~~~~
+m^{-\alpha}\sum_{k=n-N_1+2}^{n+1} \varpi_{k,q}^{(1,\alpha)}
u_{1,n-k+2}^{i+1}
\nonumber\\
&&~~~~~~~~~~~~~~
+\sum_{k=0}^{(n-N_1-1)m+N_1+1+q} w_k^{(2,\alpha)}
u_{2,(n-N_1-1)m+N_1+2-k+q}^{i+1}\bigg]
\nonumber\\
&=&u_{(n-N_1-1)m+N_1+1+q}^{i}
\nonumber\\
&&+\frac{\tau}{2}K h_2^{-\alpha}\bigg[
m^{-\alpha}\sum_{k=0}^{n-N_1+1}\varpi_{k,q}^{(1,\alpha)}
u_{1,(n-k-N_1+1)m+N_1+1}^{i}
\nonumber\\
&&~~~~~~~~~~~~~~
+m^{-\alpha}\sum_{k=n-N_1+2}^{n+1} \varpi_{k,q}^{(1,\alpha)}
u_{1,n-k+2}^{i}
\nonumber\\
&&~~~~~~~~~~~~~~
+\sum_{k=0}^{(n-N_1-1)m+N_1+1+q} w_k^{(2,\alpha)}
u_{2,(n-N_1-1)m+N_1+2-k+q}^{i}\bigg]
\nonumber\\
&&+\tau f_{(n-N_1-1)m+N_1+1+q}^{i+1/2}+O\left(\tau (h_1^{l_1}+h_2^{l_2})+\tau^3\right),
\end{eqnarray}
for $q=1,2,\cdots,m-1$, and $n=N_1+1,N_1+1+m,\cdots,N_1+1+\frac{N_I+N_2}{m}$.

Denote $U_{i}^{n}$ as the numerical approximation to $u_{i}^{n}$, and define the column vectors
\begin{eqnarray}\label{equation3.0.9}
\textbf{U}^i&=&\big[U_{1}^i,U_{2}^i,\cdots,U_{N-1}^{i}\big]^T,\label{equation3.0.10}\\
\textbf{F}^{i+1/2}&=&\big[f_{1}^{i+1/2},f_{2}^{i+1/2},\cdots,f_{N-1}^{i+1/2}\big]^T.
\label{equation3.0.11}
\end{eqnarray}
Eq. (\ref{equation3.0.3})-(\ref{equation3.0.5}) and Eq. (\ref{equation3.0.6})-(\ref{equation3.0.8}) share a general matrix form as
\begin{equation}\label{equation3.0.12}
\left(\textbf{I}-\frac{\tau K}{2}\textbf{D}\right) \textbf{U}^{i+1}
=\left(\textbf{I}+\frac{\tau K}{2}\textbf{D}\right) \textbf{U}^{i}
+\tau \textbf{F}^{i+1/2}+\textbf{H}^{i};
\end{equation}
Similar process can be carried out in Case 1 as well as in Case 2 for right Riemann-Liouville fractional derivative, just by replacing the weights, say, $g_k$, $k\downarrow$, with $g_k$, $k\uparrow$, and $(1-\frac{q}{m})$ with $\frac{q}{m}$, and $N_1$, $N_2$ with $N_2$, $N_1$. Figuratively speaking, $u_1$ or $u_2$ in Case 1 or Case 2 for left Riemann-Liouville fractional derivative can be seen as the ``transposition" of $u_2$ or $u_1$ in Case 2 or Case 1 for right Riemann-Liouville fractional derivative. Therefore, the general approximate scheme for right Riemann-Liouville problem is
\begin{equation}\label{equation3.0.13}
\left(\textbf{I}-\frac{\tau K}{2}\widetilde{\textbf{D}}\right) \textbf{U}^{i+1}
=\left(\textbf{I}+\frac{\tau K}{2}\widetilde{\textbf{D}}\right) \textbf{U}^{i}
+\tau \textbf{F}^{i+1/2}+\widetilde{\textbf{H}}^{i};
\end{equation}
see Appendix A for the concrete expressions of these matrices and vectors.

\begin{remark}\label{remark:3.1.1}
It is easy to see that if $m=1$, and $[d_{-1}^{(1)},d_0^{(1)},d_1^{(1)}]
=[d_{-1}^{(2)},d_0^{(2)},d_1^{(2)}]$, then the scheme (\ref{equation3.0.12}) turns to
the one of uniform meshes, and the differential matrix $\textbf{D}$ reduces to the
uniform finite differential matrix which has been widely discussed such as in
\cite{Li:13,Tian:12,Zhao:14,Zhou:13}, etc.
\end{remark}

\subsection{Convergence and stability analysis}\label{subsec:3.2}


Now we consider the error estimates for the general scheme (\ref{equation3.0.12}) under the following discrete $L^2$ norm defined as
\begin{equation}\label{equation3.2.1}
\|\textbf{U}\|_{h_1,h_2}=\left(h_1\sum_{n=1}^{N_1+N_I} U_n^2+h_2\sum_{n=N_1+N_I+1}^{N-1} U_n^2\right)^{1/2} ~~~\forall~ \textbf{U}\in \mathbb{R}^{N-1},
\end{equation}
if $h_1\leq h_2$; and
\begin{equation}\label{equation3.2.0}
\|\textbf{U}\|_{h_1,h_2}=\left(h_1\sum_{n=1}^{N_1} U_n^2+h_2\sum_{n=N_1+1}^{N-1} U_n^2\right)^{1/2} ~~~\forall~ \textbf{U}\in \mathbb{R}^{N-1},
\end{equation}
if $h_2\leq h_1$.

From Remark \ref{remark:3.1.1} we know that if $m=1$, and $[d_{-1}^{(1)},d_0^{(1)},d_1^{(1)}]
=[d_{-1}^{(2)},d_0^{(2)},d_1^{(2)}]$, then the scheme (\ref{equation3.0.12}) is just one of the
widely discussed scheme on uniform meshes. For different choice of $d_{-1}^{(1)}$, $d_0^{(1)}$, and $d_1^{(1)}$, we can get different scheme (\ref{equation3.0.12}). Here we are only interested in those ones whose corresponding schemes on uniform meshes are stable and convergent; that is, the following
lemma holds.

\begin{lemma}\label{lemma:3.2.1}
Let $u_{n}^{i}$ be the exact solution of problem (\ref{equation3.0.1}) and sufficiently smooth, and
$U_{n}^{i}$ the solution of a difference
scheme (\ref{equation3.0.12}) at the grid point $(x_n,t_i)$.
If $m=1$, and $[d_{-1}^{(1)},d_0^{(1)},d_1^{(1)}]=[d_{-1}^{(2)},d_0^{(2)},d_1^{(2)}]$, then for all
$1\leq i\leq M$, we have
\begin{equation}
\|\textbf{u}^i-\textbf{U}^i\|\leq c(\tau^2+h^{l}),
\end{equation}
where $h:=h_1=h_2$, $l:=l_1=l_2$, $c$ denotes a positive constant
and $\|\cdot\|$ stands for the discrete $L^2$ norm.
\end{lemma}

Now we prove the convergence of the scheme (\ref{equation3.0.12}) on non-uniform meshes.

\begin{theorem}\label{theorem3.2.1}
Fix $a$ and $\epsilon$. Let $u_{n}^{i}$ be the exact solution of problem (\ref{equation3.0.1}), and $U_{n}^{i}$ the solution of a non-uniform difference
scheme (\ref{equation3.0.12}) at the grid point $(x_n,t_i)$.
Then the following estimates
\begin{equation}\label{equation3.2.4}
\|\textbf{u}^i-\textbf{U}^i\|_{h_1,h_2}\leq c(\tau^2+h_1+h_2^{l_2}),
\end{equation}
if $h_1\leq h_2$, and
\begin{equation}\label{equation3.2.5}
\|\textbf{u}^i-\textbf{U}^i\|_{h_1,h_2}\leq c(\tau^2+h_1^{l_1}+h_2^{l_2}),
\end{equation}
if $h_2\leq h_1$, hold, where $c$ denotes a positive constant which depends
on $\alpha$, $a$, $b$ and $\epsilon$, and $\|\cdot\|_{h_1,h_2}$ stands for
the discrete $L^2$ norm defined in (\ref{equation3.2.1}) or (\ref{equation3.2.0}).
\end{theorem}

\begin{proof}
We only prove the case $h_2\leq h_1$ where $u_0(t)=0$. Similar analysis can be done
for the situation $h_1\leq h_2$.

Firstly, let $\textbf{U}^i=\textbf{U}_1^i+\textbf{U}_2^i$, where
$\textbf{U}_j^i:=\textbf{M}_j\textbf{U}^i$, and $\textbf{M}_j$, $j=1,2$,
are defined in (\ref{equation:app.0}) and (\ref{equation:app.15}). Then (\ref{equation3.0.12}) can be rewritten as
\begin{equation}\label{equation3.2._proof_1}
\begin{array}{ll}
&\big(\textbf{I}-\frac{\tau K}{2}\textbf{A}_1\big) \textbf{U}_1^{i+1}+\big(\textbf{I}-\frac{\tau K}{2}\textbf{A}_2\big) \textbf{U}_2^{i+1}\\
=&\big(\textbf{I}+\frac{\tau K}{2}\textbf{A}_1\big) \textbf{U}_1^{i}+\big(\textbf{I}+\frac{\tau K}{2}\textbf{A}_2\big) \textbf{U}_2^{i}
+\tau\textbf{F}^{i+1/2}+\textbf{H}^{i}.
\end{array}
\end{equation}
Let
\begin{equation}\label{equation3.2._proof_2}
\textbf{F}_1^{i+1/2}:=\frac{1}{\tau}\left[\big(\textbf{I}-\frac{\tau K}{2}\textbf{A}_1\big) \textbf{U}_1^{i+1}-\big(\textbf{I}+\frac{\tau K}{2}\textbf{A}_1\big) \textbf{U}_1^{i}\right],
\end{equation}
and
\begin{equation}\label{equation3.2._proof_3}
\textbf{F}_2^{i+1/2}:=\frac{1}{\tau}\left[\big(\textbf{I}-\frac{\tau K}{2}\textbf{A}_2\big) \textbf{U}_2^{i+1}-\big(\textbf{I}+\frac{\tau K}{2}\textbf{A}_2\big) \textbf{U}_2^{i}-\textbf{H}^{i}\right].
\end{equation}
Then
\begin{equation}\label{equation3.2.6}
\left\{ \begin{array}{l}
\big(\textbf{I}-\frac{\tau K}{2}\textbf{A}_1\big) \textbf{U}_1^{i+1}
=\big(\textbf{I}+\frac{\tau K}{2}\textbf{A}_1\big) \textbf{U}_1^{i}
+\tau\textbf{F}_1^{i+1/2},\\
U_{1,n}^{i}=0,~~n=N_1+N_I,\cdots,N-1,\\
U_{1,n}^0=\mathcal{M}_1^{a,\epsilon}\phi_0(\tilde{x}_n),~~n=1,\cdots,N-1,
\end{array} \right.
\end{equation}
and
\begin{equation}\label{equation3.2.7}
\left\{ \begin{array}{l}
\big(\textbf{I}-\frac{\tau K}{2}\textbf{A}_2\big) \textbf{U}_2^{i+1}
=\big(\textbf{I}+\frac{\tau K}{2}\textbf{A}_2\big) \textbf{U}_2^{i}
+\tau\textbf{F}_2^{i+1/2}+\textbf{H}^{i},\\
U_{2,n}^{i}=0,~~n=1,\cdots,N_1+1,\\
U_{2,n}^0=\mathcal{M}_2^{a,\epsilon}\phi_0(\tilde{x}_n),~~n=1,\cdots,N-1,
\end{array} \right.
\end{equation}
are numerical schemes of
\begin{equation}\label{equation3.2.proof_4}
\left\{ \begin{array}{lll}
\frac{\partial u_1(x,t) }{\partial t}&=&K\,_{0}D_x^{\alpha}u_1(x,t)+f_1(x,t) \\
&& ~~~~ ~~~~ ~~~ {\rm for}~~~ (x,t) \in (0,b)\times (0,T),\\
u_1(x,0) &=&\mathcal{M}_1^{a,\epsilon}\phi_0(x) ~~~~ {\rm for}~~~ x \in [0,b], \\
u_1(0,t)&=&0 ~~~~~~~~~~~~~~~~{\rm for}~~~ t\in[0,T],\\
u_1(b,t)&=&\mathcal{M}_1^{a,\epsilon}u_{b}(t) ~~~~~{\rm for}~~~ t\in[0,T],
\end{array} \right.
\end{equation}
and
\begin{equation}\label{equation3.2.proof_5}
\left\{ \begin{array}{lll}
\frac{\partial u_2(x,t) }{\partial t}&=&K\,_{0}D_x^{\alpha}u_2(x,t)+f_2(x,t) \\
&& ~~~~ ~~~~ ~~~ {\rm for}~~~ (x,t) \in (0,b)\times (0,T),\\
u_2(x,0) &=&\mathcal{M}_2^{a,\epsilon}\phi_0(x) ~~~~ {\rm for}~~~ x \in [0,b], \\
u_2(0,t)&=&0 ~~~~~~~~~~~~~~~~{\rm for}~~~ t\in[0,T],\\
u_2(b,t)&=&\mathcal{M}_2^{a,\epsilon}u_{b}(t) ~~~~~{\rm for}~~~ t\in[0,T],
\end{array} \right.
\end{equation}
respectively, where $f_j(x,t)$, $j=1,2$, be a function, which satisfies
\begin{eqnarray}\label{equation3.2.8}
\big[f_j(\tilde{x}_1,t_{i+1/2}),f_j(\tilde{x}_2,t_{i+1/2}),\cdots,
f_j(\tilde{x}_{N-1},t_{i+1/2})]^T&=&\textbf{F}_j^{i+1/2}.
\end{eqnarray}
Since $f_1(\tilde{x}_n,t)+f_2(\tilde{x_n},t)=f(\tilde{x}_n,t)$, $1\leq n\leq N-1$, it is clear that if $u_1(\tilde{x}_n,t)$ and $u_2(\tilde{x}_n,t)$ are the solutions of (\ref{equation3.2.proof_4}) and (\ref{equation3.2.proof_5}) at $(\tilde{x}_n,t)$, respectively, then $u(\tilde{x}_n,t)=u_1(\tilde{x}_n,t)+u_2(\tilde{x}_n,t)$, $1\leq n\leq N-1$, where $u(x,t)$ is the solution of (\ref{equation3.0.1}).

%

Denoting $\textbf{e}_2^i:=\textbf{u}_2^i-\textbf{U}_2^i$, since $e_{2,n}^{i}=0$ for $n=1,2,\cdots,N_1+1$,
by Lemma \ref{lemma:3.2.1}, there is
\begin{equation}\label{equation3.2.11}
h_2\sum_{n=N_1+2}^{N-1}|e_{2,n}^i|^2=\|e_{2}^i\|_{h_1,h_2}^2\leq c(\tau^2+h_2^{l_2})^2.
\end{equation}

For the convenience of evaluating $\textbf{e}_1^i:=\textbf{u}_1^i-\textbf{U}_1^i$, we introduce a linear interpolation
function $v(x):[0,b]\rightarrow \mathbb{R}$,
which goes through the elements of $\textbf{v}=[v_0,v_1,\cdots,v_{N}]^{T}$.
Denote
\begin{equation}\label{equation3.2.13}
\textbf{e}_{1,h_1}^i
=[u_1^i(x_1)-U_1^i(x_1),u_1^i(x_2)-U_1^i(x_2),\cdots,
u_1^i(x_{N_1+\frac{N_I+N_2}{m}})-U_1^i(x_{N_1+\frac{N_I+N_2}{m}})]^T.
\end{equation}
Because these nodes $\left\{U_1^i(x_n)\right\}_{n=1}^{N_1+\frac{N_I+N_2}{m}}$ are computed by use of a $l_1$-th order
approximation, by Lemma \ref{lemma:3.2.1}, we have
\begin{equation}\label{equation3.2.14}
h_1\sum_{n=1}^{N_1+\frac{N_I+N_2}{m}}|e_{1,h_1,n}^i|^2\leq c(\tau^2+h_1^{l_1})^2.
\end{equation}
For the errors
\begin{eqnarray}\label{equation3.2.15}
&&e_{1,h_2}^i(y_{(n-N_1-1)m-N_I+1+q})
\nonumber\\
&:=&u_1^i(y_{(n-N_1-1)m-N_I+1+q})-U_1^i(y_{(n-N_1-1)m-N_I+1+q}),
\end{eqnarray}
where $0\leq q \leq m$, $n=N_1+1,\cdots,N_1+\frac{N_I+N_2}{m}$, from Eq. (\ref{equation2.4.15}), it yields that
\begin{eqnarray}\label{equation3.2.16}
&&e_{1,h_2}^{i+1}(y_{(n-N_1-1)m-N_I+1+q})-e_{1,h_2}^i(y_{(n-N_1-1)m-N_I+1+q})
\nonumber\\
&=&\frac{\tau K}{2h_1^\alpha}\bigg[\left(1-\frac{q}{m}\right)\sum_{n=0}^{n}w_k^{(1,\alpha)}
\left(e_{1,h_1}^{i}(x_{n-k+1})+e_{1,h_1}^{i+1}(x_{n-k+1})\right)
\nonumber\\
&&~~~~~~~~~+\frac{q}{m}\sum_{n=0}^{n+1}w_k^{(1,\alpha)}
\left(e_{1,h_1}^{i}(x_{n-k+2})+e_{1,h_1}^{i+1}(x_{n-k+2})\right)\bigg]
\nonumber\\
&&+c(\tau^3+\tau h_1^{l_1}).
\end{eqnarray}
Particularly, there are
\begin{eqnarray}\label{equation3.2.17}
&&e_{1,h_2}^{i+1}(y_{(n-N_1-1)m-N_I})-e_{1,h_2}^i(y_{(n-N_1-1)m-N_I})
\nonumber\\
&=&e_{1,h_1}^{i+1}(x_{n})-e_{1,h_1}^i(x_{n})
\nonumber\\
&=&\frac{\tau K}{2h_1^\alpha}\sum_{n=0}^{n}w_k^{(1,\alpha)}
\left(e_{1,h_1}^{i}(x_{n-k+1})+e_{1,h_1}^{i+1}(x_{n-k+1})\right)
+c(\tau^3+\tau h_1^{l_1}),
\end{eqnarray}
and
\begin{eqnarray}\label{equation3.2.18}
&&e_{1,h_2}^{i+1}(y_{(n-N_1-1)m-N_I+m})-e_{1,h_2}^i(y_{(n-N_1-1)m-N_I+m})
\nonumber\\
&=&e_{1,h_1}^{i+1}(x_{n+1})-e_{1,h_1}^i(x_{n+1})
\nonumber\\
&=&\frac{\tau K}{2h_1^\alpha}\sum_{n=0}^{n+1}w_k^{(1,\alpha)}
\left(e_{1,h_1}^{i}(x_{n-k+2})+e_{1,h_1}^{i+1}(x_{n-k+2})\right)
+c(\tau^3+\tau h_1^{l_1}).
\end{eqnarray}
Thus,
\begin{eqnarray}\label{equation3.2.19}
&&e_{1,h_2}^{i+1}(y_{(n-N_1-1)m-N_I+1+q})-e_{1,h_2}^i(y_{(n-N_1-1)m-N_I+1+q})
\nonumber\\
&=&\left(1-\frac{q}{m}\right)\left[e_{1,h_1}^{i+1}(x_{n})-e_{1,h_1}^i(x_{n})\right]
+\frac{q}{m}\left[e_{1,h_1}^{i+1}(x_{n+1})-e_{1,h_1}^i(x_{n+1})\right]
\nonumber\\
&&+c(\tau^3+\tau h_1^{l_1}).
\end{eqnarray}
Summing up for all $0\leq k\leq i-1$, we have
\begin{eqnarray}\label{equation3.2.20}
&&e_{1,h_2}^i(y_{(n-N_1-1)m-N_I+1+q})
\nonumber\\
&=&\left(1-\frac{q}{m}\right)e_{1,h_1}^i(x_{n})
+\frac{q}{m}e_{1,h_1}^i(x_{n+1})+c(\tau^2+h_1^{l_1});
\end{eqnarray}
so
\begin{eqnarray}\label{equation3.2.21}
&&\left|e_{1,h_2}^i(y_{(n-N_1-1)m-N_I+1+q})\right|^2
\nonumber\\
&\leq&c\left[|e_{1,h_1}^i(x_{n})|^2
+|e_{1,h_1}^i(x_{n+1})|^2+(\tau^2+h_1^{l_1})^2\right].
\end{eqnarray}
Summing up $q$ and $n$, and multiplying by $h_2$, by the evaluation (\ref{equation3.2.14}),
Eq. (\ref{equation3.2.21}) turns to be
\begin{eqnarray}\label{equation3.2.22}
&&h_2\sum_{n=N_1+1}^{N_1+\frac{N_I+N_2}{m}}\sum_{q=1}^{m-1}
\left|e_{1,h_2}^i(y_{(n-N_1-1)m-N_I+1+q})\right|^2
\nonumber\\
&\leq&c\left[h_2\sum_{q=1}^{m-1}\sum_{n=N_1+1}^{N_1+\frac{N_I+N_2}{m}}|e_{1,h_1}^i(x_{n})|^2
+h_2\sum_{q=1}^{m-1}\sum_{n=N_1+1}^{N_1+\frac{N_I+N_2}{m}}(\tau^2+h_1^{l_1})^2\right]
\nonumber\\
&\leq&c(\tau^2+h_1^{l_1})^2.
\end{eqnarray}

Combining (\ref{equation3.2.11}), (\ref{equation3.2.14}), and (\ref{equation3.2.22}), by the definition
of the discrete norm (\ref{equation3.2.0}), we finally obtain that
\begin{eqnarray}\label{equation3.2.23}
&&\|\textbf{u}^i-\textbf{U}^i\|^2_{h_1,h_2}
\nonumber\\
&\leq&h_2\sum_{n=N_1+2}^{N-1}|e_{2,n}^i|^2
+h_1\sum_{n=1}^{N_1+\frac{N_I+N_2}{m}}|e_{1,h_1,n}^i|^2
\nonumber\\
&&+h_2\sum_{n=N_1+1}^{N_1+\frac{N_I+N_2}{m}}\sum_{q=1}^{m-1}
\left|e_{1,h_2}^i(y_{(n-N_1-1)m-N_I+1+q})\right|^2
\nonumber\\
&\leq& c(\tau^2+h_1^{l_1}+h_2^{l_2})^2,
\end{eqnarray}
which completes the proof.
\end{proof}

For the stability of the general scheme (\ref{equation3.0.12}), it can be discussed almost the same as
the convergence analysis. The details are omitted here, and the stability results are given as follows.

\begin{theorem}\label{theorem3.2.2}
Fix $a$ and $\epsilon$. The scheme (\ref{equation3.0.12}) is unconditionally stable.
\end{theorem}

Denoting the iteration matrix of (\ref{equation3.0.12}) as
$\textbf{G}=\big(\textbf{I}-\frac{\tau K}{2}\textbf{D}\big)^{-1}\big(\textbf{I}+\frac{\tau K}{2}\textbf{D}\big)$,
Figs. \ref{fig:3.2.1}-\ref{fig:3.2.4} numerically show that the spectral radius of $\textbf{G}$ is indeed no bigger than $1$ for all
$\alpha\in(1,2)$, and different partitions, with $a=\epsilon=1$, $m=2$, $b=4$, and $\tau=1/400$.

\begin{figure}[!htbp]
\includegraphics[scale=0.4]{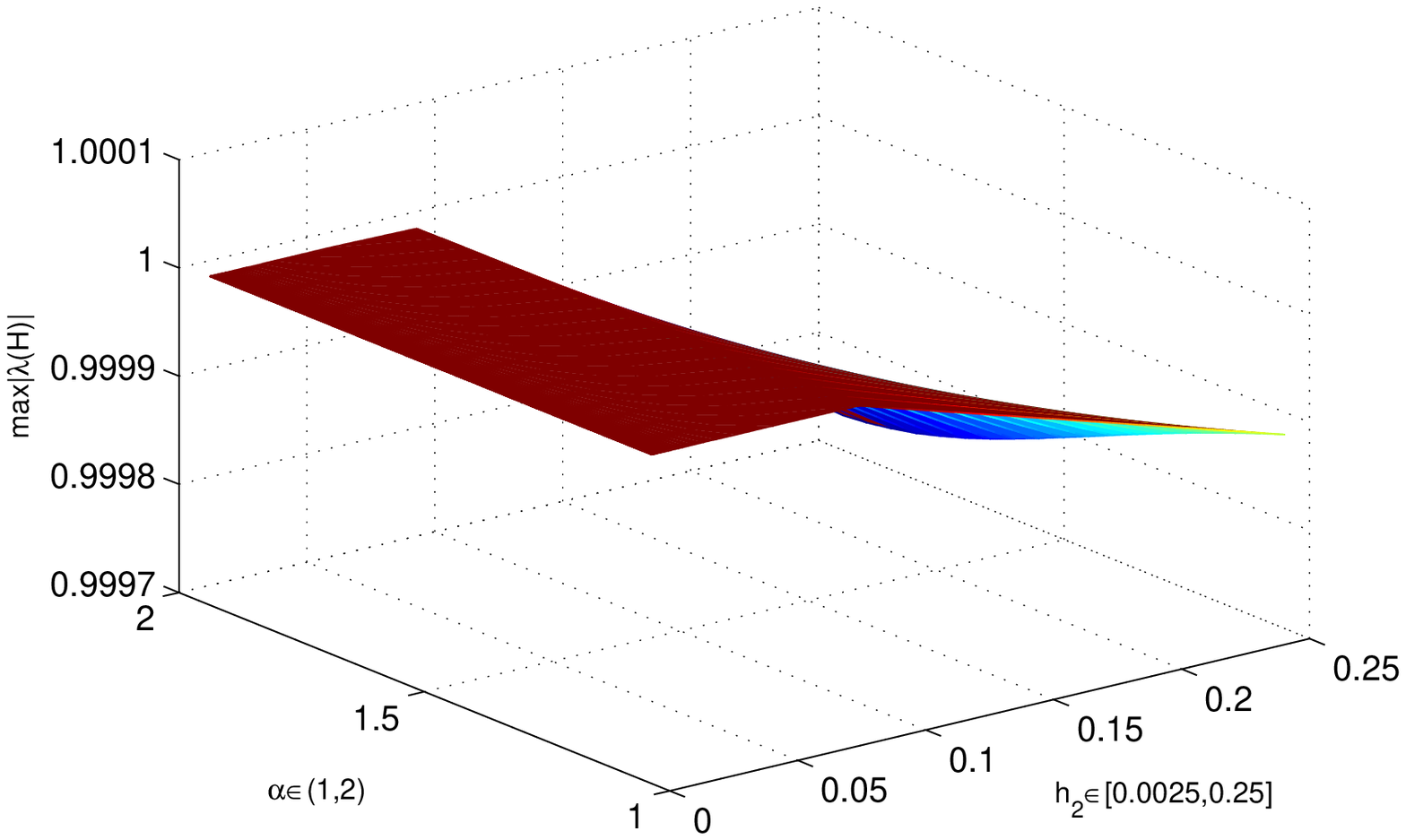}\\
\caption{The relationship among $\max|\lambda(\textbf{G})|$,
$\alpha$ and $h_2$ for the scheme (\ref{equation3.0.12})  in which $l_1=l_2=1$ and
$[d_{-1}^{(1)},d_0^{(1)},d_1^{(1)}]=[d_{-1}^{(2)},d_0^{(2)},d_1^{(2)}]=[0,0,1]$ on the partition Case 1
of Subsection \ref{subsec:2.2} with $a=\epsilon=1$, $m=2$, $b=4$, and $\tau=1/400$.}\label{fig:3.2.1}
\end{figure}

\begin{figure}[!htbp]
\includegraphics[scale=0.4]{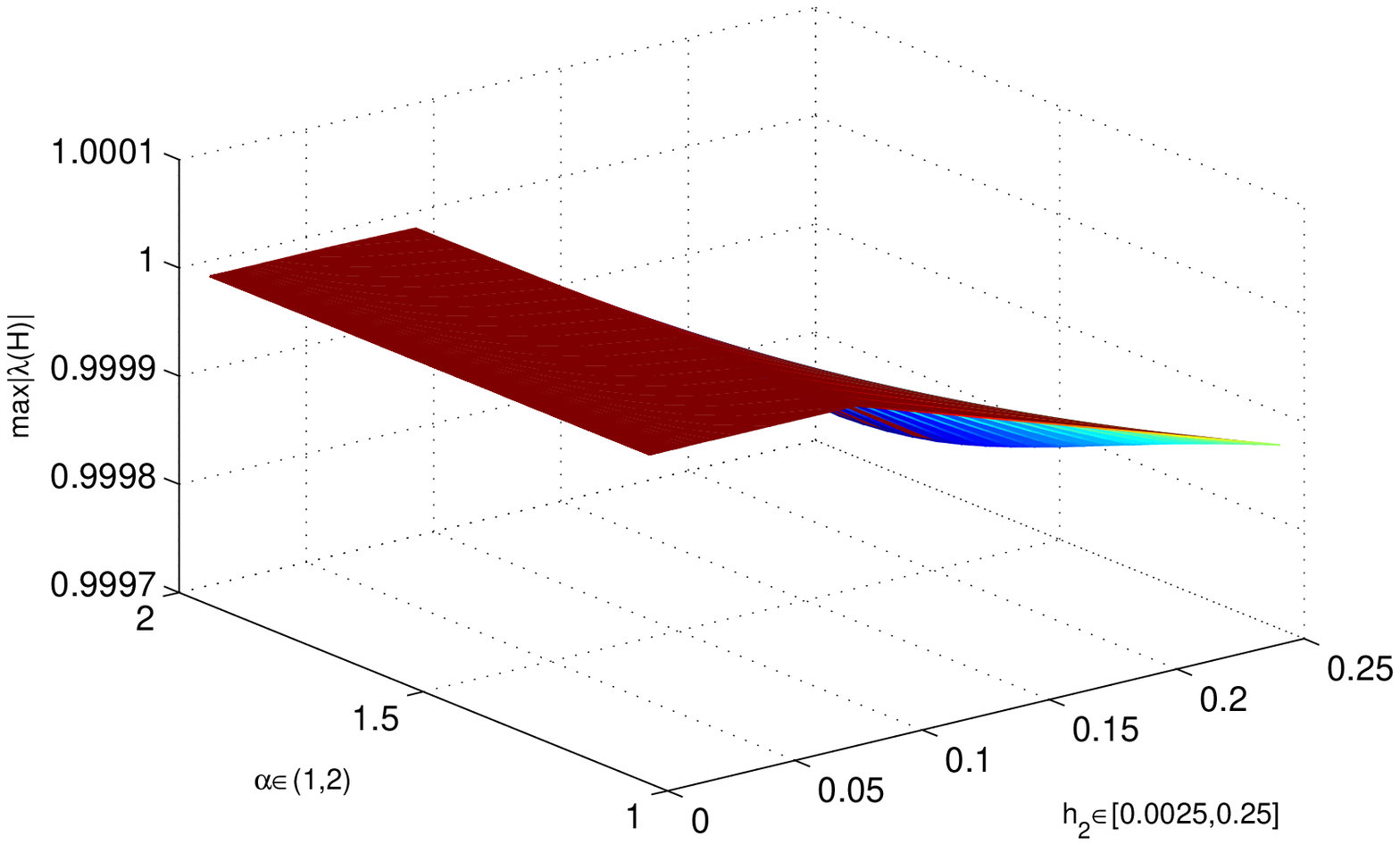}\\
\caption{The relationship among $\max|\lambda(\textbf{G})|$,
$\alpha$ and $h_1$ for the scheme (\ref{equation3.0.12}) in which $l_1=l_2=1$ and
$[d_{-1}^{(1)},d_0^{(1)},d_1^{(1)}]=[d_{-1}^{(2)},d_0^{(2)},d_1^{(2)}]=[0,0,1]$ on the partitions Case 2
of Subsection \ref{subsec:2.2} with $a=\epsilon=1$, $m=2$, $b=4$, and $\tau=1/400$.}\label{fig:3.2.2}
\end{figure}


\begin{figure}[!htbp]
\includegraphics[scale=0.4]{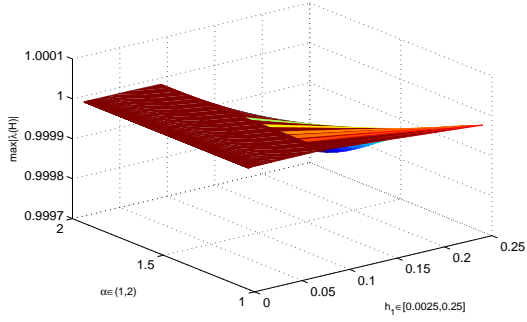}\\
\caption{The relationship among $\max|\lambda(\textbf{G})|$,
$\alpha$ and $h_1$ for the scheme (\ref{equation3.0.12}) in which $l_1=l_2=2$ and
$[d_{-1}^{(1)},d_0^{(1)},d_1^{(1)}]=[d_{-1}^{(2)},d_0^{(2)},d_1^{(2)}]=[0,1-\frac{\alpha}{2},\frac{\alpha}{2}]$
on the partitions Case 2
of Subsection \ref{subsec:2.2} with $a=\epsilon=1$, $m=2$, $b=4$, and $\tau=1/400$.}\label{fig:3.2.4}
\end{figure}

\section{Numerical Experiments}\label{sec:4}

We perform numerical experiments in this section to confirm the theoretical analysis and convergence orders.
For the convenience of the following presentation, we denote $(11+11)$ as the scheme (\ref{equation3.0.12}) which combines
two first order shifted Gr\"{u}nwald  formula \cite{Meerschaert:04}, i.e.,
$[d_{-1}^{(1)},d_0^{(1)},d_1^{(1)}]=[d_{-1}^{(2)},d_0^{(2)},d_1^{(2)}]=[0,0,1]$;
denote $(21+21)$ as the scheme (\ref{equation3.0.12}) which combines two second order finite difference schemes
with $[d_{-1}^{(1)},d_0^{(1)},d_1^{(1)}]=[d_{-1}^{(2)},d_0^{(2)},d_1^{(2)}]=[0,1-\frac{\alpha}{2},\frac{\alpha}{2}]$;
denote $(22+22)$ as the scheme (\ref{equation3.0.12}) which combines two second order finite difference schemes
with $[d_{-1}^{(1)},d_0^{(1)},d_1^{(1)}]=[d_{-1}^{(2)},d_0^{(2)},d_1^{(2)}]
=[\frac{2-\alpha}{4},0,\frac{2+\alpha}{4}]$.
\begin{example}\label{example4.0.2}
Consider the following problem
\begin{equation}\label{equation4.0.2}
\frac{\partial u(x,t)}{\partial t}=
\,_{a}D_{x}^{\alpha}u(x,t)+f(x,t),
~~~(x,t)\in(0,4)\times(0,1],
\end{equation}
with the source term
\begin{eqnarray*}
&&f(x,t)
\nonumber\\
&=&-e^{-t}\big(x^{2+\alpha}+x^2+\frac{\Gamma(3+\alpha)}{2}x^2+\frac{2}{\Gamma(3-\alpha)}x^{2-\alpha}\big),
\end{eqnarray*}
and the boundary conditions
\begin{equation}
u(0,t)=0,~~~u(4,t)=e^{-t}(16+4^{2+\alpha}),~~~t\in[0,1],
\end{equation}
and the initial value
\begin{equation}
u(x,0)=x^{2+\alpha}+x^2,~~~x\in[0,4].
\end{equation}
Then the exact solutions of (\ref{equation4.0.2}) is $e^{-t}\left(x^{2+\alpha}+x^2\right)$.
\end{example}

By applying non-uniform schemes with $\tau=1/400$, and $h_2=h_1/2$ for scheme $(11+11)$, $(21+21)$ as well as $(22+22)$, we can see from Table \ref{table4.0.5} that the optimal convergence orders can be obtained.

\begin{table}
\caption{The discrete $L^2$ errors ($e_{h_1,h_2}=\|u-U\|_{h_1,h_2}$) and their convergence rates to Example \ref{example4.0.2} at $t=1$
by using several specific schemes on non-uniform meshes
for different $\alpha$ with $a=\epsilon=1$, $h_2=h_1/2$, and $\tau=1/400$.}\label{table4.0.5}
\begin{tabular}{cccccccccc}
\hline\noalign{\smallskip}
\multicolumn{2}{c}{$\alpha$} &\multicolumn{2}{c}{$1.2$} &\multicolumn{2}{c}{$1.4$}
&\multicolumn{2}{c}{$1.6$}&\multicolumn{2}{c}{$1.8$}\\
\noalign{\smallskip}\hline\noalign{\smallskip}
scheme &$h_1$ & $e_{h_1,h_2}$ & rate & $e_{h_1,h_2}$ & rate
& $e_{h_1,h_2}$ & rate &  $e_{h_1,h_2}$ & rate\\
\noalign{\smallskip}\hline\noalign{\smallskip}
             & 1/8 & 5.13 1e-1 & -      & 4.35 1e-1 & -   & 3.47 1e-1 & -   & 2.38 1e-1 & -   \\
             & 1/16& 2.58 1e-1 & 0.99   & 2.07 1e-1 & 1.00& 1.71 1e-1 & 1.02& 1.07 1e-1 & 1.15 \\
(11+11)      & 1/32& 1.29 1e-1 & 1.00   & 1.08 1e-1 & 1.00& 8.49 1e-2 & 1.01& 5.21 1e-2 & 1.04 \\
             & 1/64& 6.47 1e-2 & 1.00   & 5.42 1e-2 & 1.00& 4.24 1e-2 & 1.00& 2.59 1e-2 & 1.01 \\
             &1/128& 3.24 1e-2 & 1.00   & 2.71 1e-2 & 1.00& 2.12 1e-2 & 1.00& 1.29 1e-2 & 1.00 \\
\noalign{\smallskip}\hline\noalign{\smallskip}
             & 1/8 & 4.90 1e-2 & -      & 5.26 1e-2 & -   & 6.79 1e-2 & -   & 1.07 1e-1 & -   \\
             & 1/16& 1.27 1e-2 & 1.95   & 1.41 1e-2 & 1.90& 1.85 1e-2 & 1.88& 2.70 1e-2 & 1.99 \\
(21+21)      & 1/32& 3.22 1e-3 & 1.98   & 3.51 1e-3 & 2.01& 4.50 1e-3 & 2.04& 6.49 1e-3 & 2.05 \\
             & 1/64& 8.09 1e-4 & 1.99   & 8.53 1e-4 & 2.04& 1.05 1e-3 & 2.10& 1.44 1e-3 & 2.17 \\
             &1/128& 2.06 1e-4 & 1.98   & 2.09 1e-4 & 2.03& 2.42 1e-4 & 2.12& 3.03 1e-4 & 2.25 \\
\noalign{\smallskip}\hline\noalign{\smallskip}
             & 1/8 & 8.71 1e-2 & -      & 7.48 1e-2 & -   & 8.20 1e-2 & -   & 1.14 1e-1 & -   \\
             & 1/16& 2.18 1e-2 & 2.00   & 1.92 1e-2 & 1.96& 2.10 1e-2 & 1.97& 2.73 1e-2 & 2.06 \\
(22+22)      & 1/32& 5.45 1e-3 & 2.00   & 4.81 1e-3 & 2.00& 5.20 1e-3 & 2.01& 6.70 1e-3 & 2.03 \\
             & 1/64& 1.36 1e-3 & 2.00   & 1.19 1e-2 & 2.02& 1.24 1e-3 & 2.07& 1.50 1e-3 & 2.16 \\
             &1/128& 3.44 1e-4 & 1.99   & 2.95 1e-4 & 2.01& 2.94 1e-4 & 2.07& 3.23 1e-4 & 2.21 \\
\noalign{\smallskip}\hline
\end{tabular}
\end{table}


\begin{example}\label{example4.0.3}
Consider the following problem
\begin{equation}\label{equation4.0.3}
\frac{\partial u(x,t)}{\partial t}=
\,_{a}D_{x}^{\alpha}u(x,t)+f,~~~(x,t)\in(0,4)\times(0,1],
\end{equation}
with the source term
\begin{eqnarray*}
&&f(x,t)
\nonumber\\
&=&-e^{-t}\bigg(\frac{\Gamma(1+\frac{\alpha}{4})}{\Gamma(1-\frac{3\alpha}{4})}x^{-3\alpha/4}
+\frac{1}{\Gamma(1-\alpha)}x^{-\alpha}
+x^{\alpha/4}+1\bigg),
\end{eqnarray*}
and the boundary conditions
\begin{equation}
u(0,t)=e^{-t},~~~u(4,t)=\left(4^{\alpha/4}+1 \right)e^{-t},~~~t\in[0,1],
\end{equation}
and the initial value
\begin{equation}
u(x,0)=x^{\alpha/4}+1,~~~x\in[0,4].
\end{equation}
Then the exact solution of (\ref{equation4.0.3}) is $e^{-t}\left(x^{\alpha/4}+1\right)$.
\end{example}

Table \ref{table4.0.6} shows that by using the usual shifted Gr\"{u}nwald  formula on uniform meshes, the convergence orders of
the approximations to the solution of Example \ref{example4.0.3} are less than one.
While by applying the non-uniform scheme $(11+11)$, with $a=\epsilon=1$, $h_1=h_2/5$, and $\tau=1/400$, we can see
that errors caused by $h_2$ are dominant, and thus the desired first order convergence is reached.

\begin{table}
\caption{The discrete $L^2$ errors  ($e_{h_1,h_2}=\|u-U\|_{h_1,h_2}$) and their convergence rates to Example \ref{example4.0.3} at $t=1$
by using uniform as well as non-uniform first order methods
for different $\alpha$, where $a=\epsilon=1$ and $\tau=1/400$.}\label{table4.0.6}
\begin{tabular}{cccccccccc}
\hline\noalign{\smallskip}
\multicolumn{2}{c}{$\alpha$}   &\multicolumn{2}{c}{1.2} &\multicolumn{2}{c}{1.4}
       &\multicolumn{2}{c}{1.6} &\multicolumn{2}{c}{1.8}\\
\noalign{\smallskip}\hline\noalign{\smallskip}
$scheme$&$h_2$& $e_{h_1,h_2}$& rate & $e_{h_1,h_2}$ & rate & $e_{h_1,h_2}$ & rate & $e_{h_1,h_2}$ & rate\\
\noalign{\smallskip}\hline\noalign{\smallskip}
           & 1/8 & 1.13 1e-2 & -      & 1.09 1e-2 & -   & 1.46 1e-2 & -   & 2.56 1e-2 & -   \\
S-G-L      & 1/16& 6.98 1e-3 & 0.70   & 6.51 1e-3 & 0.74& 8.84 1e-3 & 0.73& 1.66 1e-2 & 0.63 \\
$(h_1=h_2)$& 1/32& 4.19 1e-3 & 0.74   & 3.80 1e-3 & 0.78& 5.23 1e-3 & 0.76& 1.07 1e-2 & 0.64 \\
           &1/64 & 2.47 1e-3 & 0.77   & 2.18 1e-3 & 0.80& 3.05 1e-3 & 0.78& 6.79 1e-3 & 0.65 \\
           &1/128& 1.45 1e-3 & 0.77   & 1.24 1e-3 & 0.81& 1.76 1e-3 & 0.79& 4.30 1e-3 & 0.66 \\
\noalign{\smallskip}\hline\noalign{\smallskip}
             & 1/8 & 2.22 1e-2 & -      & 1.62 1e-2 & -   & 2.19 1e-2 & -   & 5.09 1e-2 & -   \\
(11+11)      & 1/16& 1.12 1e-2 & 0.99   & 7.33 1e-3 & 1.15& 6.00 1e-3 & 1.87& 9.39 1e-3 & 2.44 \\
$(h_1=h_2/5)$& 1/32& 5.55 1e-3 & 1.01   & 3.43 1e-3 & 1.10& 2.60 1e-3 & 1.21& 4.62 1e-3 & 1.02 \\
             & 1/64& 2.79 1e-3 & 0.99   & 1.71 1e-3 & 1.00& 1.26 1e-3 & 1.04& 2.47 1e-3 & 0.91 \\
             &1/128& 1.41 1e-3 & 0.98   & 8.68 1e-4 & 0.98& 6.63 1e-4 & 0.93& 1.51 1e-3 & 0.71 \\
\noalign{\smallskip}\hline
\end{tabular}
\end{table}


\begin{example}\label{example4.0.5}
Consider the following problem
\begin{equation}\label{equation4.0.5}
\frac{\partial u(x,t)}{\partial t}=
\,_{a}D_{x}^{\alpha}u(x,t)+f(x,t),
~~~(x,t)\in(0,4)\times(0,1],
\end{equation}
with the source term
\begin{eqnarray*}
&&f(x,t)
\nonumber\\
&=&-e^{-t}\big(x^{1+\alpha}+x+\Gamma(2+\alpha)x+\frac{1}{\Gamma(2-\alpha)}x^{1-\alpha}\big),
\end{eqnarray*}
and the boundary conditions
\begin{equation}
u(0,t)=0,~~~u(4,t)=e^{-t}(4+4^{1+\alpha}),~~~t\in[0,1],
\end{equation}
and the initial value
\begin{equation}
u(x,0)=x^{1+\alpha}+x,~~~x\in[0,4].
\end{equation}
Then the exact solutions of (\ref{equation4.0.5}) is $e^{-t}\left(x^{1+\alpha}+x\right)$.
\end{example}


From Table \ref{table4.0.8} we can see that by using second order schemes on uniform meshes,
the convergence order of the approximations to the solution of Example \ref{example4.0.5} can not reach to 2. While by applying non-uniform schemes with $\tau=1/400$, and $h_1=h_2/5$ for schemes $(21+21)$ as well as $(22+22)$, we can see from Table \ref{table4.0.9} that errors caused by $h_2$ are dominant, and thus the optimal convergence orders can be obtained.

\begin{table}
\caption{The discrete $L^2$ errors and their convergence rates to Example \ref{example4.0.5} at $t=1$
by using specific schemes on uniform meshes
for different $\alpha$ with $\tau=1/400$.}\label{table4.0.8}
\begin{tabular}{cccccccccc}
\hline\noalign{\smallskip}
\multicolumn{2}{c}{$\alpha$} &\multicolumn{2}{c}{$1.2$} &\multicolumn{2}{c}{$1.4$}
&\multicolumn{2}{c}{$1.6$}&\multicolumn{2}{c}{$1.8$}\\
\noalign{\smallskip}\hline\noalign{\smallskip}
scheme&$h$ & $\|u-U\|$ & rate & $\|u-U\|$ & rate& $\|u-U\|$ & rate & $\|u-U\|$ & rate\\
\noalign{\smallskip}\hline\noalign{\smallskip}
       & 1/8  & 1.12 1e-3 & -    & 1.83 1e-3 & -    & 2.49 1e-3 & -    & 2.50 1e-3 & - \\
       & 1/16 & 3.72 1e-4 & 1.59 & 6.62 1e-4 & 1.47 & 9.87 1e-4 & 1.34 & 1.11 1e-3 & 1.18 \\
21     & 1/32 & 1.28 1e-4 & 1.53 & 2.42 1e-4 & 1.45 & 3.91 1e-4 & 1.34 & 4.90 1e-4 & 1.17\\
       & 1/64 & 4.57 1e-5 & 1.49 & 8.89 1e-5 & 1.44 & 1.55 1e-4 & 1.33 & 2.18 1e-4 & 1.17\\
       & 1/128& 1.79 1e-5 & 1.35 & 3.38 1e-5 & 1.39 & 6.25 1e-5 & 1.31 & 9.79 1e-5 & 1.15\\
\noalign{\smallskip}\hline\noalign{\smallskip}
       & 1/8  & 3.42 1e-3 & -    & 2.03 1e-3 & -    & 2.17 1e-3 & -    & 2.38 1e-3 & - \\
       & 1/16 & 1.25 1e-3 & 1.45 & 7.67 1e-4 & 1.40 & 8.72 1e-4 & 1.32 & 1.05 1e-3 & 1.17 \\
22     & 1/32 & 4.49 1e-4 & 1.48 & 2.83 1e-4 & 1.44 & 3.53 1e-4 & 1.30 & 4.73 1e-4 & 1.16\\
       & 1/64 & 1.60 1e-5 & 1.49 & 1.04 1e-4 & 1.45 & 1.43 1e-4 & 1.30 & 2.13 1e-4 & 1.15\\
       & 1/128& 5.74 1e-5 & 1.48 & 3.90 1e-5 & 1.42 & 5.88 1e-5 & 1.28 & 9.65 1e-5 & 1.14\\
\noalign{\smallskip}\hline
\end{tabular}
\end{table}

\begin{table}
\caption{The discrete $L^2$ errors ($e_{h_1,h_2}=\|u-U\|_{h_1,h_2}$) and their convergence rates to Example \ref{example4.0.5} at $t=1$
by using specific schemes on non-uniform meshes
for different $\alpha$ with $a=\epsilon=1$, $h_1=h_2/5$ and $\tau=1/400$.}\label{table4.0.9}
\begin{tabular}{cccccccccc}
\hline\noalign{\smallskip}
\multicolumn{2}{c}{$\alpha$} &\multicolumn{2}{c}{$1.2$} &\multicolumn{2}{c}{$1.4$}
&\multicolumn{2}{c}{$1.6$}&\multicolumn{2}{c}{$1.8$}\\
\noalign{\smallskip}\hline\noalign{\smallskip}
scheme &$h_2$ & $e_{h_1,h_2}$ & rate & $e_{h_1,h_2}$ & rate
& $e_{h_1,h_2}$ & rate &  $e_{h_1,h_2}$ & rate\\
\noalign{\smallskip}\hline\noalign{\smallskip}
\noalign{\smallskip}\hline\noalign{\smallskip}
             & 1/8 & 6.18 1e-2 & -      & 9.13 1e-2 & -   & 1.90 1e-1 & -   & 5.78 1e-1 & -   \\
(21+21)      & 1/16& 2.28 1e-2 & 1.44   & 3.76 1e-2 & 1.28& 5.70 1e-2 & 1.73& 9.40 1e-2 & 2.62 \\
             & 1/32& 5.50 1e-3 & 2.05   & 8.43 1e-3 & 2.15& 1.46 1e-2 & 1.96& 3.51 1e-2 & 1.42 \\
             & 1/64& 1.27 1e-3 & 2.11   & 1.72 1e-3 & 2.29& 3.34 1e-3 & 2.13& 8.50 1e-3 & 2.05 \\
             &1/128& 2.77 1e-4 & 2.20   & 3.42 1e-4 & 2.33& 7.95 1e-4 & 2.07& 1.88 1e-3 & 2.18 \\
\noalign{\smallskip}\hline\noalign{\smallskip}
             & 1/8 & 8.01 1e-2 & -      & 1.02 1e-1 & -   & 1.90 1e-1 & -   & 5.70 1e-1 & -   \\
(22+22)      & 1/16& 2.57 1e-2 & 1.64   & 3.84 1e-2 & 1.41& 5.66 1e-2 & 1.75& 9.32 1e-2 & 2.61 \\
             & 1/32& 6.16 1e-3 & 2.06   & 8.69 1e-3 & 2.14& 1.47 1e-2 & 1.94& 3.50 1e-2 & 1.41 \\
             & 1/64& 1.43 1e-3 & 2.11   & 1.79 1e-3 & 2.28& 3.36 1e-3 & 2.13& 8.48 1e-3 & 2.04 \\
             &1/128& 3.20 1e-4 & 2.16   & 3.64 1e-4 & 2.30& 7.98 1e-4 & 2.08& 1.88 1e-3 & 2.18 \\
\noalign{\smallskip}\hline
\end{tabular}
\end{table}
\begin{example}\label{example4.0.4}
Consider the following problem
\begin{equation}\label{equation4.0.4}
\frac{\partial u(x,t)}{\partial t}=
\,_{a}D_{x}^{\alpha}u(x,t)+f(x,t),
~~~(x,t)\in(0,4)\times(0,1],
\end{equation}
with the source term
\begin{eqnarray*}
&&f(x,t)
\nonumber\\
&=&-e^{-t}\big(x^{1+|\alpha-1.5|/2}
+\frac{\Gamma(2+|\alpha-1.5|/2)}{\Gamma(2-\alpha+|\alpha-1.5|/2)}x^{1-\alpha+|\alpha-1.5|/2}\big),
\end{eqnarray*}
and the boundary conditions
\begin{equation}
u(0,t)=0,~~~u(4,t)=e^{-t}(4^{1+|\alpha-1.5|/2}),~~~t\in[0,1],
\end{equation}
and the initial value
\begin{equation}
u(x,0)=x^{1+|\alpha-1.5|/2},~~~x\in[0,4].
\end{equation}
Then the exact solutions of (\ref{equation4.0.4}) is $e^{-t}(x^{1+|\alpha-1.5|/2})$.
\end{example}

From Table \ref{table4.0.7} we can see that by using second order schemes on uniform meshes,
the convergence order of the approximations to the solution of Example \ref{example4.0.4} can not reach to 2. While by applying non-uniform schemes with $\tau=1/400$, and $h_1=h_2/5$ for schemes $(21+21)$ as well as $(22+22)$, we can see from Table \ref{table4.0.8} that errors caused by $h_2$
are dominant, and thus the optimal convergence orders can be obtained.

\begin{table}
\caption{The discrete $L^2$ errors and their convergence rates to Example \ref{example4.0.4} at $t=1$
by using specific schemes on uniform meshes
for different $\alpha$ with $\tau=1/400$.}\label{table4.0.7}
\begin{tabular}{cccccccccc}
\hline\noalign{\smallskip}
\multicolumn{2}{c}{$\alpha$} &\multicolumn{2}{c}{$1.2$} &\multicolumn{2}{c}{$1.4$}
&\multicolumn{2}{c}{$1.6$}&\multicolumn{2}{c}{$1.8$}\\
\noalign{\smallskip}\hline\noalign{\smallskip}
scheme&$h$ & $\|u-U\|$ & rate & $\|u-U\|$ & rate& $\|u-U\|$ & rate & $\|u-U\|$ & rate\\
\noalign{\smallskip}\hline\noalign{\smallskip}
       & 1/8  & 8.19 1e-4 & -    & 1.25 1e-3 & -    & 2.09 1e-3 & -    & 2.30 1e-3 & - \\
       & 1/16 & 2.41 1e-4 & 1.76 & 4.57 1e-4 & 1.45 & 8.27 1e-4 & 1.34 & 9.70 1e-4 & 1.25\\
21     & 1/32 & 7.41 1e-5 & 1.70 & 1.65 1e-4 & 1.47 & 3.22 1e-4 & 1.36 & 3.99 1e-4 & 1.28\\
       & 1/64 & 2.32 1e-5 & 1.68 & 5.89 1e-5 & 1.48 & 1.24 1e-4 & 1.38 & 1.62 1e-4 & 1.30\\
       & 1/128& 7.44 1e-6 & 1.64 & 2.10 1e-5 & 1.49 & 4.73 1e-5 & 1.39 & 6.53 1e-5 & 1.31\\
\noalign{\smallskip}\hline\noalign{\smallskip}
       & 1/8  & 3.92 1e-3 & -    & 2.10 1e-3 & -    & 1.68 1e-3 & -    & 2.00 1e-3 & - \\
       & 1/16 & 1.19 1e-3 & 1.72 & 7.12 1e-4 & 1.56 & 7.09 1e-4 & 1.25 & 8.90 1e-4 & 1.17 \\
22     & 1/32 & 3.70 1e-4 & 1.69 & 2.47 1e-4 & 1.53 & 2.87 1e-4 & 1.30 & 3.78 1e-4 & 1.23\\
       & 1/64 & 1.16 1e-4 & 1.67 & 8.59 1e-5 & 1.52 & 1.13 1e-4 & 1.34 & 1.56 1e-4 & 1.27\\
       & 1/128& 3.69 1e-5 & 1.66 & 3.00 1e-5 & 1.52 & 4.41 1e-5 & 1.36 & 6.37 1e-5 & 1.30\\
\noalign{\smallskip}\hline
\end{tabular}
\end{table}

\begin{table}
\caption{The discrete $L^2$ errors ($e_{h_1,h_2}=\|u-U\|_{h_1,h_2}$) and their convergence rates to Example \ref{example4.0.4} at $t=1$
by using specific schemes on non-uniform meshes
for different $\alpha$ with $a=\epsilon=1$, $h_1=h_2/5$, and $\tau=1/400$.}\label{table4.0.8}
\begin{tabular}{cccccccccc}
\hline\noalign{\smallskip}
\multicolumn{2}{c}{$\alpha$} &\multicolumn{2}{c}{$1.2$} &\multicolumn{2}{c}{$1.4$}
&\multicolumn{2}{c}{$1.6$}&\multicolumn{2}{c}{$1.8$}\\
\noalign{\smallskip}\hline\noalign{\smallskip}
scheme &$h_1$ & $e_{h_1,h_2}$ & rate & $e_{h_1,h_2}$ & rate
& $e_{h_1,h_2}$ & rate &  $e_{h_1,h_2}$ & rate\\
\noalign{\smallskip}\hline\noalign{\smallskip}
             & 1/8 & 1.73 1e-2 & -      & 1.97 1e-2 & -   & 3.49 1e-2 & -   & 9.54 1e-2 & -   \\
(21+21)      & 1/16& 6.10 1e-3 & 1.51   & 7.37 1e-3 & 1.42& 9.25 1e-3 & 1.92& 1.50 1e-2 & 2.67 \\
             & 1/32& 1.47 1e-3 & 2.05   & 1.66 1e-3 & 2.15& 2.47 1e-3 & 1.91& 5.53 1e-3 & 1.44 \\
             & 1/64& 3.38 1e-4 & 2.12   & 3.42 1e-4 & 2.28& 5.76 1e-4 & 2.10& 1.33 1e-3 & 2.06 \\
             &1/128& 7.39 1e-5 & 2.19   & 6.96 1e-5 & 2.30& 1.38 1e-4 & 2.06& 2.93 1e-4 & 2.18 \\
\noalign{\smallskip}\hline\noalign{\smallskip}
             & 1/8 & 2.31 1e-2 & -      & 2.24 1e-2 & -   & 3.52 1e-2 & -   & 9.43 1e-2 & -   \\
(22+22)      & 1/16& 7.07 1e-3 & 1.70   & 7.67 1e-3 & 1.55& 9.26 1e-3 & 1.93& 1.50 1e-2 & 2.65 \\
             & 1/32& 1.70 1e-3 & 2.06   & 1.74 1e-3 & 2.14& 2.49 1e-3 & 1.89& 5.52 1e-3 & 1.44 \\
             & 1/64& 3.95 1e-4 & 2.10   & 3.65 1e-4 & 2.26& 5.83 1e-4 & 2.10& 1.33 1e-3 & 2.05 \\
             &1/128& 8.93 1e-5 & 2.14   & 7.65 1e-5 & 2.25& 1.39 1e-4 & 2.06& 2.92 1e-4 & 2.19 \\
\noalign{\smallskip}\hline
\end{tabular}
\end{table}
\section{Conclusions} \label{sec:5}
It is well known that most of the time putting more nodes around singularities and less nodes in smooth region is required in numerically solving PDEs. However, for this kind of meshes, it is not easy to get a convergent finite difference scheme of space fractional operators. This paper provides a basic strategy to overcome this challenge: based on the technique of mollification, to be approximated solution is firstly decomposed into the exact sum of some functions which are no lower regular than the original function and have relatively small common support; then the functions are discretized with different stepsizes; adding all the discretizations of the functions leads to the approximation of the solution on non-uniform meshes; when solving the equation, may the techniques of interpolation be needed in the approximation. In fact, by this way, it is very flexible to design numerical schemes because of the independence of the mollified functions; similar to the h-p finite element methods, both the h approximation (different stepsizes are used) and the p approximation (low and high order schemes are simultaneously used) work well. The proposed schemes, including the h and p approximations, are theoretically and numerically discussed in detail. We rigorously prove their optimal convergence and unconditional stability. The extensive numerical experiments are performed to show the powerfulness of the schemes and confirm the theoretical results. The ideas given in this paper are expected to stimulate more research in the numerical solutions of fractional PDEs.

\appendix
\renewcommand{\appendixname}{Appendix~\Alph{section}}
\section*{Appendix A}

\begin{enumerate}

\item{Elements in the scheme of left Riemann-Liouville problem}\\

Here we show the specific expressions of the matrices and vectors in (\ref{equation3.0.12}), where
\begin{eqnarray}
\textbf{D}=\textbf{D}_1+\textbf{D}_2,\label{equation:app.1}\\
\textbf{D}_1=\textbf{A}_1\textbf{M}_1,\label{equation:app.2}\\
\textbf{D}_2=\textbf{A}_2\textbf{M}_2,\label{equation:app.3}\\
\textbf{M}_1+\textbf{M}_2=\textbf{I}\label{equation:app.0}.
\end{eqnarray}

If $h_1\leq h_2$, then

\begin{equation}\label{equation:app.7}
\begin{array}{l}\textbf{A}_1=h_1^{-\alpha}\cdot
\left[ \begin{array}{cccccc;{2pt/2pt}ccc}
g_1^{(\alpha)} & g_0^{(\alpha)} & &&&&&&           \\
g_2^{(\alpha)}  & g_1^{(\alpha)} & g_0^{(\alpha)}&&&&&        \\
 \ddots &\ddots  &  \ddots    &   \ddots  &&&&\\
g_{N_1+N_I-1}^{(\alpha)}&g_{N_1+N_I-2}^{(\alpha)}&\cdots &
g_{2}^{(\alpha)}&g_{1}^{(\alpha)}&g_{0}^{(\alpha)}&0&\cdots&0 \\
g_{N_1+N_I}^{(\alpha)}&g_{N_1+N_I-1}^{(\alpha)}&\cdots &
g_{3}^{(\alpha)}&g_{2}^{(\alpha)}&g_{1}^{(\alpha)}&0&\cdots&0 \\
\hdashline[2pt/2pt]
g_{N_1+N_I+m}^{(\alpha)}&g_{N_1+N_I+m-1}^{(\alpha)}&\cdots &g_{m+3}^{(\alpha)}&g_{m+2}^{(\alpha)}
&g_{m+1}^{(\alpha)}&0&\cdots&0\\
\ddots &\ddots  &     &   \ddots  &\ddots&\ddots&\vdots& &\vdots\\
g_{N_1+N_I+N_2 m}^{(\alpha)}&
g_{N_1+N_I+N_2 m-1}^{(\alpha)}&\cdots
&g_{N_2 m+3}^{(\alpha)}&g_{N_2 m+2}^{(\alpha)}
&g_{N_2m+1}^{(\alpha)}&0&\cdots&0\\
\end{array}
\right]\\
\begin{array}{cc}
\rule{100mm}{0mm}&\underbrace{\rule{10mm}{0mm}}_{N_2}
\end{array}
\end{array};
\end{equation}

\begin{equation}\label{equation:app.8}
\begin{array}{c}
\textbf{A}_2=h_2^{-\alpha}\cdot
\left[ \begin{array}{c}
\textbf{A}_{2,1}\\
\textbf{A}_{2,2}\\
\vdots\\
\textbf{A}_{2,\frac{N_1+N_I}{m}-1}\\
\textbf{A}_{2,\frac{N_1+N_I}{m}}\\
\textbf{A}_{2,\frac{N_1+N_I}{m}+1}\\
\end{array}
\right],
\end{array}
\end{equation}
where
\begin{equation}
\begin{array}{l}
\textbf{A}_{2,k}=
\left[ \begin{array}{cccccccccccccccccccccccc}
0&\cdots&0&\varpi_{k,1}^{(2,\alpha)}& 0&\cdots&0&\varpi_{k-1,1}^{(2,\alpha)}&0&\cdots&0&\cdots&
0&\cdots&0&\varpi_{1,1}^{(2,\alpha)}&0&\cdots&0 &\varpi_{0,1}^{(2,\alpha)}&0&\cdots&0\\
0&\cdots&0&\varpi_{k,2}^{(2,\alpha)}& 0&\cdots&0&\varpi_{k-1,2}^{(2,\alpha)}&0&\cdots&0&\cdots&
0&\cdots&0&\varpi_{1,2}^{(2,\alpha)}&0&\cdots&0 &\varpi_{0,2}^{(2,\alpha)}&0&\cdots&0\\
 &\vdots& & \vdots         &  &\vdots& & \vdots         &  &\vdots& &\vdots&  &\vdots& &\vdots&  &\vdots&& \vdots&& \vdots\\
0&\cdots&0&\varpi_{k,m}^{(2,\alpha)}& 0&\cdots&0&\varpi_{k-1,m}^{(2,\alpha)}&0&\cdots&0&\cdots&
0&\cdots&0&\varpi_{1,m}^{(2,\alpha)}&0&\cdots&0 &\varpi_{0,m}^{(2,\alpha)}&0&\cdots&0\\
\end{array}
\right]_{m\times (N-1)}\\
\begin{array}{cccccccccccc}
\rule{11mm}{0mm}&
\underbrace{\rule{9mm}{0mm}}_{m-1}&\rule{9mm}{0mm}&
\underbrace{\rule{9mm}{0mm}}_{m-1}&\rule{10mm}{0mm}&
\underbrace{\rule{9mm}{0mm}}_{m-1}&\rule{4mm}{0mm}&
\underbrace{\rule{9mm}{0mm}}_{m-1}&\rule{9mm}{0mm}&
\underbrace{\rule{9mm}{0mm}}_{m-1}&\rule{5mm}{0mm}&
\underbrace{\rule{9mm}{0mm}}_{N-1-(k+1)m}
\end{array}
\end{array}
\end{equation}
for $k=1,2,\cdots,\frac{N_1+N_I}{m}-1$, and
\begin{equation}
\begin{array}{l}
\textbf{A}_{2,\frac{N_1+N_I}{m}}=
\\
\left[ \begin{array}{cccccccccccccccccccccccc}
0&\cdots&0& \varpi_{\frac{N_1+N_I}{m},1}^{(2,\alpha)}& 0&\cdots&0&\cdots&
0&\cdots&0&\varpi_{2,1}^{(2,\alpha)}&0&\cdots&0 &\varpi_{1,1}^{(2,\alpha)}&\varpi_{0,1}^{(2,\alpha)}& 0&\cdots&0\\
0&\cdots&0& \varpi_{\frac{N_1+N_I}{m},2}^{(2,\alpha)}& 0&\cdots&0&\cdots&
0&\cdots&0&\varpi_{2,2}^{(2,\alpha)}&0&\cdots&0 &\varpi_{1,2}^{(2,\alpha)}&\varpi_{0,2}^{(2,\alpha)}& 0&\cdots&0\\
 &\vdots& & \vdots         &  &\vdots& & \vdots         &  &\vdots& &\vdots&  &\vdots& &\vdots& &\vdots \\
0&\cdots&0& \varpi_{\frac{N_1+N_I}{m},m}^{(2,\alpha)}& 0&\cdots&0&\cdots&
0&\cdots&0&\varpi_{2,m}^{(2,\alpha)}&0&\cdots&0 &\varpi_{1,m}^{(2,\alpha)}&\varpi_{0,m}^{(2,\alpha)}& 0&\cdots&0\\
\end{array}
\right]_{m\times (N-1)}\\
\begin{array}{cccccccccccc}
\rule{1mm}{0mm}&
\underbrace{\rule{9mm}{0mm}}_{m-1}&\rule{15mm}{0mm}&
\underbrace{\rule{9mm}{0mm}}_{m-1}&\rule{4mm}{0mm}&
\underbrace{\rule{9mm}{0mm}}_{m-1}&\rule{9mm}{0mm}&
\underbrace{\rule{9mm}{0mm}}_{m-1}&\rule{18mm}{0mm}&
\underbrace{\rule{9mm}{0mm}}_{N_2-1}
\end{array}
\end{array},
\end{equation}
and
\begin{equation}
\begin{array}{l}
\textbf{A}_{2,\frac{N_1+N_I}{m}+1}=
\\
\left[ \begin{array}{cccccccccccccccccccccccc}
0&\cdots&0& w_{\frac{N_1+N_I}{m}+1}^{(2,\alpha)}& 0&\cdots&0&\cdots&
0&\cdots&0&w_{2}^{(2,\alpha)}&w_{1}^{(2,\alpha)}&w_{0}^{(2,\alpha)}\\
 &\vdots& &\vdots                           &
 &\vdots& &\vdots                           &&\vdots& &\ddots&\ddots&\ddots\\
0&\cdots&0 &w_{\frac{N_1+N_I}{m}+N_2-1}^{(2,\alpha)}&
0&\cdots&0&\cdots&0&\cdots&0 &w_{N_2}^{(2,\alpha)}&w_{N_2-1}^{(2,\alpha)}&
  \cdots&w_1^{(2,\alpha)}&w_0^{(2,\alpha)}\\
0&\cdots&0 &w_{\frac{N_1+N_I}{m}+N_2}^{(2,\alpha)}&0&\cdots&0&\cdots&
0&\cdots&0 &w_{N_2+1}^{(2,\alpha)}&w_{N_2}^{(2,\alpha)}&
  \cdots&w_2^{(2,\alpha)}&w_1^{(2,\alpha)}\\
\end{array}
\right]_{N_2 \times (N-1)}\\
\begin{array}{cccccccccc}
\rule{1mm}{0mm}&
\underbrace{\rule{9mm}{0mm}}_{m-1}&\rule{21mm}{0mm}&
\underbrace{\rule{9mm}{0mm}}_{m-1}&\rule{4mm}{0mm}&
\underbrace{\rule{9mm}{0mm}}_{m-1}&\rule{37mm}{0mm}
\end{array}
\end{array};
\end{equation}

\begin{equation}\label{equation:app.4}
\textbf{M}_1=
\left[ \begin{array}{ccccccccccccc}
1\\
&1\\
&&\ddots\\
&&&1\\
\hdashline[2pt/2pt]
&&&&s_0^{h_1,\epsilon}\\
&&&&&s_1^{h_1,\epsilon}\\
&&&&&&\ddots\\
&&&&&&&&s_{N_I-1}^{h_1,\epsilon}\\
\hdashline[2pt/2pt]
&&&&&&&&&&0\\
&&&&&&&&&&&\ddots\\
&&&&&&&&&&&&0
\end{array}
\right]\begin{array}{l}
\left.\rule{0mm}{9mm}\right\}N_1\\
\left.\rule{0mm}{10mm}\right\} N_I\\
\left.\rule{0mm}{8mm}\right\} N_2
\end{array},
\end{equation}
where
\begin{equation}\label{equation:app.6}
s_k^{h_1,\epsilon}=\int_{-1}^{1-k h_1/\epsilon}\rho(z)dz, ~~k=0,1,\cdots,N_I-1;
\end{equation}

\begin{equation}\label{equation:app.9}
\textbf{H}^{i}=\textbf{H}_0^{i}+\textbf{H}_N^{i},
\end{equation}
where
\begin{eqnarray}
\textbf{H}_0^{i}=\frac{\tau K}{2}
\left(h_1^{-\alpha}\textbf{h}_{0}^D-\textbf{h}_{0}^S\right)\left(u_0^i+u_0^{i+1}\right),
\label{equation:app.10}\\
\textbf{H}_N^{i}=\frac{\tau K}{2h_2^{\alpha}}
\textbf{h}_{N}^D\left(u_b^i+u_b^{i+1}\right),\label{equation:app.11}
\end{eqnarray}

\begin{equation}\label{equation:app.12}
\textbf{h}_{0}^D(n)=\left\{ \begin{array}{ll}
\frac{n^{-\alpha}}{\Gamma(1-\alpha)},
&1\leq n\leq N_1+N_I-1,\\
\\
\frac{\big(N_1+N_I+(n-N_1-N_I)m\big)^{-\alpha}}{\Gamma(1-\alpha)},
&N_1+N_I \leq n\leq N-1;
\end{array}
\right.
\end{equation}

\begin{equation}\label{equation:app.13}
\textbf{h}_{0}^S=\textbf{D}\left[1,1,\cdots,1\right]^{T}+h_2^{-\alpha}\left[0,0,\cdots,0,w_0^{(2,\alpha)}\right]^T;
\end{equation}

and
\begin{eqnarray}\label{equation:app.14}
\textbf{h}_{N}^D=\left[0,0,\cdots,0,w_0^{(2,\alpha)}\right]^T.
\end{eqnarray}


If $h_2\leq h_1$, then
\begin{equation}\label{equation:app.18}
\begin{array}{l}
\textbf{A}_1=h_1^{-\alpha}\cdot
\left[ \begin{array}{c}
\textbf{A}_{1,1}\\
\textbf{A}_{1,2}\\
\vdots\\
\textbf{A}_{1,\frac{N_1+N_I}{m}-1}\\
\textbf{A}_{1,\frac{N_1+N_I}{m}}\\
\textbf{A}_{1,\frac{N_1+N_I}{m}+1}\\
\end{array}
\right]
\end{array},
\end{equation}
where
\begin{equation}
\begin{array}{l}
\textbf{A}_{1,1}=
\left[ \begin{array}{cccccccccccccccccccccccc}
w_{1}^{(1,\alpha)}   & w_{0}^{(1,\alpha)}&0&0&0&\cdots&0    \\
w_{2}^{(1,\alpha)}   & w_{1}^{(1,\alpha)}    & w_{0}^{(1,\alpha)}&0&0&\cdots&0  \\
\ddots  &\ddots    &  \ddots              \\
w_{N_1}^{(1,\alpha)} & w_{N_1-1}^{(1,\alpha)}& \cdots  &
w_{0}^{(1,\alpha)}     &0&\cdots&0 \\
\end{array}
\right]_{N_1 \times (N-1)}\\
\begin{array}{cc}
\rule{48mm}{0mm}&
\underbrace{\rule{9mm}{0mm}}_{N-2-N_1}
\end{array}
\end{array},
\end{equation}
and
\begin{equation}
\begin{array}{l}
\textbf{A}_{1,k}=\\
\left[ \begin{array}{cccccccccccccccccccccccc}
\varpi_{N_1+k,0}^{(1,\alpha)}&
\cdots&\varpi_{k,0}^{(1,\alpha)}& 0&\cdots&0&\cdots& 0&\cdots&0& \varpi_{1,0}^{(1,\alpha)}& 0&\cdots&0&\varpi_{0,0}^{(1,\alpha)}& 0&\cdots&0\\
\varpi_{N_1+k,1}^{(1,\alpha)}&
\cdots&\varpi_{k,1}^{(1,\alpha)}& 0&\cdots&0&\cdots& 0&\cdots&0& \varpi_{1,1}^{(1,\alpha)}& 0&\cdots&0&\varpi_{0,1}^{(1,\alpha)}& 0&\cdots&0\\
\ddots&\ddots&\vdots&&\vdots&&\vdots&&\vdots&&\vdots&&\vdots&&\vdots&&\vdots\\
\varpi_{N_1+k,m-1}^{(1,\alpha)}&
\cdots&\varpi_{k,m-1}^{(1,\alpha)}& 0&\cdots&0&\cdots& 0&\cdots&0& \varpi_{1,m-1}^{(1,\alpha)}& 0&\cdots&0&\varpi_{0,m-1}^{(1,\alpha)}& 0&\cdots&0\\
\end{array}
\right]_{m\times (N-1)}\\
\begin{array}{cccccccccc}
\rule{34mm}{0mm}&
\underbrace{\rule{9mm}{0mm}}_{m-1}&\rule{5mm}{0mm}&
\underbrace{\rule{9mm}{0mm}}_{m-1}&\rule{10mm}{0mm}&
\underbrace{\rule{9mm}{0mm}}_{m-1}&\rule{7mm}{0mm}&
\underbrace{\rule{9mm}{0mm}}_{N-2-N_1-km}
\end{array}
\end{array}
\end{equation}
for $k=2,\cdots,\frac{N_1+N_I}{m}-1$,
and
\begin{equation}
\begin{array}{l}
\textbf{A}_{1,\frac{N_1+N_I}{m}}=\\
\left[ \begin{array}{cccccccccccccccccccccccc}
\varpi_{N_1+\frac{N_I+N_2}{m},0}^{(1,\alpha)}&
\cdots&\varpi_{\frac{N_I+N_2}{m},0}^{(1,\alpha)}& 0&\cdots&0&\cdots& 0&\cdots&0& \varpi_{2,0}^{(1,\alpha)}& 0&\cdots&0&\varpi_{1,0}^{(1,\alpha)}&0&\cdots&0\\
\varpi_{N_1+\frac{N_I+N_2}{m},1}^{(1,\alpha)}&
\cdots&\varpi_{\frac{N_I+N_2}{m},1}^{(1,\alpha)}& 0&\cdots&0&\cdots& 0&\cdots&0& \varpi_{2,1}^{(1,\alpha)}& 0&\cdots&0&\varpi_{1,1}^{(1,\alpha)}&0&\cdots&0\\
\ddots&\ddots&\vdots&&\vdots&&\vdots&&\vdots&&\vdots&&\vdots&&\vdots&&\vdots\\
\varpi_{N_1+\frac{N_I+N_2}{m},m-1}^{(1,\alpha)}&
\cdots&\varpi_{\frac{N_I+N_2}{m},m-1}^{(1,\alpha)}& 0&\cdots&0&\cdots& 0&\cdots&0& \varpi_{2,m-1}^{(1,\alpha)}& 0&\cdots&0&\varpi_{1,m-1}^{(1,\alpha)}&0&\cdots&0\\
\end{array}
\right]_{m\times (N-1)}\\
\begin{array}{cccccccccc}
\rule{50mm}{0mm}&
\underbrace{\rule{9mm}{0mm}}_{m-1}&\rule{4mm}{0mm}&
\underbrace{\rule{9mm}{0mm}}_{m-1}&\rule{11mm}{0mm}&
\underbrace{\rule{9mm}{0mm}}_{m-1}&\rule{10mm}{0mm}&
\underbrace{\rule{9mm}{0mm}}_{m-1}
\end{array}
\end{array},
\end{equation}
\begin{equation}
\begin{array}{l}
\textbf{A}_{1,\frac{N_1+N_I}{m}+1}=\\
\left[ \begin{array}{cccccccccccccccccccccccc}
\varpi_{N_1+\frac{N_I+N_2}{m}+1,0}^{(1,\alpha)}&
\cdots&\varpi_{\frac{N_I+N_2}{m}+1,0}^{(1,\alpha)}& 0&\cdots&0&\cdots& 0&\cdots&0& \varpi_{3,0}^{(1,\alpha)}& 0&\cdots&0&\varpi_{2,0}^{(1,\alpha)}& 0&\cdots&0\\
\varpi_{N_1+\frac{N_I+N_2}{m}+1,1}^{(1,\alpha)}&
\cdots&\varpi_{\frac{N_I+N_2}{m}+1,1}^{(1,\alpha)}& 0&\cdots&0&\cdots& 0&\cdots&0& \varpi_{3,1}^{(1,\alpha)}& 0&\cdots&0&\varpi_{2,1}^{(1,\alpha)}& 0&\cdots&0\\
\ddots&\ddots&\vdots&&\vdots&&\vdots&&\vdots&&\vdots&&\vdots&&\vdots&&\vdots\\
\varpi_{N_1+\frac{N_I+N_2}{m}+1,m-1}^{(1,\alpha)}&
\cdots&\varpi_{\frac{N_I+N_2}{m}+1,m-1}^{(1,\alpha)}& 0&\cdots&0&\cdots& 0&\cdots&0& \varpi_{3,m-1}^{(1,\alpha)}& 0&\cdots&0&\varpi_{2,m-1}^{(1,\alpha)}& 0&\cdots&0\\
\end{array}
\right]_{m\times (N-1)}\\
\begin{array}{cccccccccc}
\rule{57mm}{0mm}&
\underbrace{\rule{9mm}{0mm}}_{m-1}&\rule{4mm}{0mm}&
\underbrace{\rule{9mm}{0mm}}_{m-1}&\rule{11mm}{0mm}&
\underbrace{\rule{9mm}{0mm}}_{m-1}&\rule{10mm}{0mm}&
\underbrace{\rule{9mm}{0mm}}_{m-1}
\end{array}
\end{array};
\end{equation}

\begin{equation}\label{equation:app.20}
\begin{array}{l}
\textbf{A}_2=h_2^{-\alpha}\cdot
\left[ \begin{array}{ccc;{2pt/2pt}cccccc}
0&\cdots&0&0 &\cdots&\cdots&\cdots&\cdots&0\\
&\vdots&&\vdots&&&&&\vdots\\
0&\cdots&0&0 &\cdots&\cdots&\cdots&\cdots&0\\
\hdashline[2pt/2pt]
0&\cdots&0& w_1^{(2,\alpha)} & w_0^{(2,\alpha)}          \\
0&\cdots&0& w_2^{(2,\alpha)}  & w_1^{(2,\alpha)} & w_0^{(2,\alpha)} \\
&\vdots & & \ddots &\ddots  &  \ddots    &   \ddots \\
0&\cdots&0& w_{N_2+N_I-1}^{(2,\alpha)}&w_{N_2+N_I-2}^{(2,\alpha)}&\cdots &
w_{2}^{(2,\alpha)}&w_{1}^{(2,\alpha)}&w_{0}^{(2,\alpha)}\\
0&\cdots&0& w_{N_2+N_I}^{(2,\alpha)}&w_{N_2+N_I-1}^{(2,\alpha)}&\cdots &
w_{3}^{(2,\alpha)}&w_{2}^{(2,\alpha)}&w_{1}^{(2,\alpha)}\\
\end{array}
\right]
\begin{array}{l}
\left.\rule{0mm}{7mm}\right\}N_1\\
\rule{0mm}{21mm}
\end{array}
\\
\begin{array}{cc}
\rule{17mm}{0mm}&\underbrace{\rule{9mm}{0mm}}_{N_1}
\end{array}
\end{array};
\end{equation}

\begin{equation}\label{equation:app.15}
\textbf{M}_1=
\left[ \begin{array}{ccccccccccccc}
1\\
&1\\
&&\ddots\\
&&&1\\
\hdashline[2pt/2pt]
&&&&s_0^{h_2,\epsilon}\\
&&&&&s_1^{h_2,\epsilon}\\
&&&&&&\ddots\\
&&&&&&&&s_{N_I-1}^{h_2,\epsilon}\\
\hdashline[2pt/2pt]
&&&&&&&&&&0\\
&&&&&&&&&&&\ddots\\
&&&&&&&&&&&&0
\end{array}
\right]\begin{array}{l}
\left.\rule{0mm}{9mm}\right\}N_1\\
\left.\rule{0mm}{10mm}\right\} N_I\\
\left.\rule{0mm}{8mm}\right\} N_2
\end{array},
\end{equation}
where
\begin{equation}\label{equation:app.16}
s_k^{h_2,\epsilon}=\int_{-1}^{1-k h_2/\epsilon}\rho(z)dz, ~~k=0,1,\cdots,N_I-1;
\end{equation}

\begin{eqnarray}\label{equation:app.25}
\textbf{H}^{i}=\frac{\tau K}{2h_2^{\alpha}}
\textbf{h}_{N}^D\left(u_b^i+u_b^{i+1}\right),
\end{eqnarray}
and
\begin{eqnarray}\label{equation:app.28}
\textbf{h}_{N}^D=\left[0,0,\cdots,0,w_0^{(2,\alpha)}\right]^T.
\end{eqnarray}

\item{Elements in the scheme of right Riemann-Liouville problem}\\

As for the matrices and vectors in (\ref{equation3.0.13}), there is
\begin{eqnarray}
\widetilde{\textbf{D}}=\widetilde{\textbf{D}}_1+\widetilde{\textbf{D}}_2,\label{equation:app.2.2.1}\\
\widetilde{\textbf{D}}_1=\widetilde{\textbf{A}}_1\textbf{M}_1,\label{equation:app.2.2.2}\\
\widetilde{\textbf{D}}_2=\widetilde{\textbf{A}}_2\textbf{M}_2.\label{equation:app.2.2.3}
\end{eqnarray}

If $h_1\leq h_2$, then
\begin{eqnarray}
\widetilde{\textbf{A}}_1=\widehat{\textbf{A}}_2^T,\label{equation:app.2.2.4}\\
\widetilde{\textbf{A}}_2=\widehat{\textbf{A}}_1^T,\label{equation:app.2.2.5}
\end{eqnarray}
where $\widehat{\textbf{A}}_1$ and $\widehat{\textbf{A}}_2$ are same like $\textbf{A}_1$ in (\ref{equation:app.18}) and $\textbf{A}_2$ in (\ref{equation:app.20}), respectively, just by replacing $N_1$ with $N_2$, $N_2$ with $N_1$, $(1-\frac{q}{m})$ with $\frac{q}{m}$, $h_1$ with $h_2$, and $h_2$ with $h_1$.
And
\begin{eqnarray}\label{equation:app.2.2.6}
\widetilde{\textbf{H}}^{i}=\frac{\tau K}{2h_1^{\alpha}}
\widetilde{\textbf{h}}_{0}^D\left(u_0^i+u_0^{i+1}\right),
\end{eqnarray}
where
\begin{eqnarray}\label{equation:app.2.2.7}
\widetilde{\textbf{h}}_{0}^D=\left[w_0^{(2,\alpha)},0,0,\cdots,0\right]^T.
\end{eqnarray}

If $h_2\leq h_1$, then
\begin{eqnarray}
\widetilde{\textbf{A}}_1=\widehat{\textbf{A}}_2^T,\label{equation:app.2.2.8}\\
\widetilde{\textbf{A}}_2=\widehat{\textbf{A}}_1^T,\label{equation:app.2.2.9}
\end{eqnarray}
where $\widehat{\textbf{A}}_1$ and $\widehat{\textbf{A}}_2$ are same like $\textbf{A}_1$ in (\ref{equation:app.7}) and $\textbf{A}_2$ in (\ref{equation:app.8}), respectively, just by replacing $N_1$ with $N_2$, $N_2$ with $N_1$, $(1-\frac{q}{m})$ with $\frac{q}{m}$, $h_1$ with $h_2$, and $h_2$ with $h_1$.
And

\begin{equation}\label{equation:app.2.2.9}
\widetilde{\textbf{H}}^{i}=\widetilde{\textbf{H}}_0^{i}+\widetilde{\textbf{H}}_N^{i},
\end{equation}
where
\begin{eqnarray}
\widetilde{\textbf{H}}_N^{i}=\frac{\tau K}{2}
\left(h_2^{-\alpha}\widetilde{\textbf{h}}_{N}^D-\widetilde{\textbf{h}}_{N}^S\right)\left(u_b^i+u_b^{i+1}\right),
\label{equation:app.2.2.14}\\
\widetilde{\textbf{H}}_0^{i}=\frac{\tau K}{2h_1^{\alpha}}
\widetilde{\textbf{h}}_{0}^D\left(u_0^i+u_0^{i+1}\right),\label{equation:app.2.2.10}
\end{eqnarray}

\begin{equation}\label{equation:app.2.2.11}
\widetilde{\textbf{h}}_{N}^D(n)=\left\{ \begin{array}{ll}
\frac{\big(N_2+N_I+(N_1+1-n)m\big)^{-\alpha}}{\Gamma(1-\alpha)},
&1\leq n\leq N_1+1,\\
\frac{(N-n)^{-\alpha}}{\Gamma(1-\alpha)},
&N_1+2\leq n\leq N-1;\\
\end{array}
\right.
\end{equation}

\begin{equation}\label{equation:app.2.2.12}
\widetilde{\textbf{h}}_{N}^S=\widetilde{\textbf{D}}\left[1,1,\cdots,1\right]^{T}+h_1^{-\alpha}\left[w_0^{(2,\alpha)},0,0,\cdots,0\right]^T;
\end{equation}
and
\begin{eqnarray}\label{equation:app.2.2.13}
\widetilde{\textbf{h}}_{0}^D=\left[w_0^{(2,\alpha)},0,0,\cdots,0\right]^T.
\end{eqnarray}

\end{enumerate}

\section*{Appendix~B}

During the real computation, perturbation on matrix \textbf{D} in Eq. (66) should be taken into consideration, since $\int_{-1}^{1-d(x)/\epsilon}\rho(z)dz$ is numerically computed as in Remark \ref{remark:2.2.1}. Now we give a brief estimate for the difference $\textbf{U}^{i+1}-\widetilde{\textbf{U}}^{i+1}$, where $\widetilde{\textbf{U}}^{i+1}$ is the solution of the system
\begin{equation}\label{equation_1_remark}
\left(\textbf{I}-\frac{\tau K}{2}\left(\textbf{D}+\delta\textbf{D}\right)\right) \widetilde{\textbf{U}}^{i+1}
=\left(\textbf{I}+\frac{\tau K}{2}\left(\textbf{D}+\delta\textbf{D}\right)\right) \textbf{U}^{i}
+\tau \textbf{F}^{i+1/2}+\textbf{H}^{i},
\end{equation}
and $\delta\textbf{D}$ is the perturbation matrix of $\textbf{D}$, assuming that as $\textbf{I}-\frac{\tau K}{2}\textbf{D}$, $\textbf{I}-\frac{\tau K}{2}\left(\textbf{D}+\delta\textbf{D}\right)$ is nonsingular and real, too.

Here, we need a basic result about error estimates and condition numbers as in the following lemma:
\begin{lemma}(\cite{Ortega:87})\label{lemma:3.2.2}
Assume that $\textbf{x}$ and $\textbf{y}$ satisfies
\begin{equation}\label{equation_2_remark}
\textbf{A}\textbf{x}=\textbf{b}
\end{equation}
and
\begin{equation}\label{equation_3_remark}
\left(\textbf{A}+\delta\textbf{A}\right)\textbf{y}=\textbf{b}+\delta\textbf{b},
\end{equation}
respectively, where $\textbf{A}$ and $\delta\textbf{A}$ are real $(N-1)\times(N-1)$ matrices, $\textbf{b}$ and $\delta\textbf{b}$ are real $(N-1)\times 1$ vectors. If $\|\textbf{A}^{-1}\|\cdot \|\delta\textbf{A}\|<1$, then $\textbf{A}+\delta\textbf{A}$ is nonsingular, and
\begin{equation}\label{equation_4_remark}
\frac{\|\textbf{x}-\textbf{y}\|}{\|\textbf{x}\|}\leq \frac{K(\textbf{A})}{1-K(\textbf{A})\|\delta\textbf{A}\|/\|\textbf{A}\|}
\left(\frac{\|\delta\textbf{A}\|}{\|\textbf{A}\|}+\frac{\|\delta\textbf{b}\|}{\|\textbf{b}\|}\right),
\end{equation}
where $K(\textbf{A}):=\|\textbf{A}^{-1}\|\cdot\|\textbf{A}\|$ is the condition number of $\textbf{A}$.
\end{lemma}

It is obvious that $(\ref{equation3.0.12})-(\ref{equation_1_remark})$ leads to
\begin{equation}\label{equation_5_remark}
\left(\textbf{I}-\frac{\tau K}{2}\textbf{D}\right) \left(\textbf{U}^{i+1}-\widetilde{\textbf{U}}^{i+1}\right)
+\frac{\tau K}{2}\delta\textbf{D}\cdot\widetilde{\textbf{U}}^{i+1}
=-\frac{\tau K}{2}\delta\textbf{D}\cdot\textbf{U}^{i},
\end{equation}
and $(\ref{equation_2_remark})-(\ref{equation_3_remark})$ leads to
\begin{equation}\label{equation_6_remark}
\textbf{A}\left(\textbf{x}-\textbf{y}\right)
-\delta\textbf{A}\cdot\textbf{y}=-\delta\textbf{b}.
\end{equation}
Comparing Eq. (\ref{equation_5_remark}) with Eq. (\ref{equation_6_remark}), and using Lemma \ref{lemma:3.2.2}, we can get the following error estimate for the difference between the solutions of $(\ref{equation3.0.12})$ and $(\ref{equation_1_remark})$.

\begin{lemma}\label{remark3.2.2}
If If $\|\left(\textbf{I}-\frac{\tau K}{2}\textbf{D}\right)^{-1}\|\cdot \|\frac{\tau K}{2}\delta\textbf{D}\|<1$, then $\textbf{I}-\frac{\tau K}{2}\left(\textbf{D}+\delta\textbf{D}\right)$ is nonsingular, and
\begin{equation}\label{equation_7_remark}
\frac{\|\textbf{U}^{i+1}-\widetilde{\textbf{U}}^{i+1}\|}{\|\textbf{U}^{i+1}\|}\leq \frac{K(\textbf{L})}{1-K(\textbf{L})\|\delta\textbf{L}\|/\|\textbf{L}\|}
\left(\frac{\|\delta\textbf{L}\|}{\|\textbf{L}\|}+\frac{\|\delta\textbf{v}^i\|}{\|\textbf{v}^i\|}\right)
\end{equation}
holds, where $\delta\textbf{L}=-\frac{\tau K}{2}\delta\textbf{D}$, $\textbf{L}=\textbf{I}-\frac{\tau K}{2}\textbf{D}$,
$\delta\textbf{v}^i=\frac{\tau K}{2}\delta\textbf{D}\cdot\textbf{U}^{i}$, and $\textbf{v}^i=\left(\textbf{I}+\frac{\tau K}{2}\textbf{D}\right) \textbf{U}^{i}
+\tau \textbf{F}^{i+1/2}+\textbf{H}^{i}$.
\end{lemma}

Table \ref{table_1_remark} lists sets of condition numbers of $\textbf{L}=\textbf{I}-\frac{\tau K}{2}\textbf{D}$ in the sense of 2-norm for several schemes with $a=\epsilon=1$, $h_1=h_2/5$, $b=4$, and $\tau=1/400$. Results for $h_1>h_2$ are similar. Here $(11)+(11)$, $(21)+(21)$, and $(22)+(22)$ present the schemes (\ref{equation3.0.12}) that combine with specific formulae as state in numerical section.

\begin{table}
\caption{The condition numbers of $\textbf{L}$ and the preconditioned matrix $\textbf{P}^{-1}\textbf{L}$ for several specific schemes on non-uniform meshes
and different $\alpha$ with $a=\epsilon=1$, $h_1=h_2/5$, and $\tau=1/400$.}\label{table_1_remark}
\begin{tabular}{cccccccccc}
\hline\noalign{\smallskip}
\multicolumn{2}{c}{$\alpha$} &\multicolumn{2}{c}{$1.2$} &\multicolumn{2}{c}{$1.4$}&\multicolumn{2}{c}{$1.6$}&\multicolumn{2}{c}{$1.8$}\\
\noalign{\smallskip}\hline\noalign{\smallskip}
scheme &$h_2$  & $K(\textbf{L})$ & $K(\textbf{P}^{-1}\textbf{L})$
& $K(\textbf{L})$ & $K(\textbf{P}^{-1}\textbf{L})$
& $K(\textbf{L})$ & $K(\textbf{P}^{-1}\textbf{L})$
& $K(\textbf{L})$ & $K(\textbf{P}^{-1}\textbf{L})$ \\
\noalign{\smallskip}\hline\noalign{\smallskip}
             & 1/8 & 1.25 & 1.04 & 1.60 & 1.05 & 2.43 & 1.08 & 4.45 &1.12\\
             & 1/16& 1.58 & 1.09 & 2.62 & 1.15 & 5.53 & 1.26 &13.96 &1.44 \\
(11+11)      & 1/32& 2.39 & 1.21 & 5.52 & 1.43 & 16.26& 1.88 &54.02 &2.88\\
             & 1/64& 4.37 & 1.52 &14.36 & 2.28 & 56.01& 4.57 &226.33&13.46\\
             &1/128& 9.58 & 2.32 &42.46 & 5.58 &199.39& 23.16&911.40&131.36\\
\noalign{\smallskip}\hline\noalign{\smallskip}
             & 1/8 & 1.06 & 1.02 & 1.24 & 1.04 & 1.86 & 1.06 & 3.75 &1.11\\
             & 1/16& 1.15 & 1.06 & 1.69 & 1.12 & 3.71 & 1.22 &11.28 &1.41 \\
(21+21)      & 1/32& 1.42 & 1.14 & 3.07 & 1.33 & 10.07& 1.77 &42.96 &2.76\\
             & 1/64& 2.24 & 1.35 & 7.44 & 2.03 & 33.79& 4.12 &180.27&12.56\\
             &1/128& 4.68 & 1.95 &21.67 & 4.78 &120.30& 20.04&728.43&123.01\\
\noalign{\smallskip}\hline\noalign{\smallskip}
             & 1/8 & 1.25 & 1.03 & 1.59 & 1.05 & 2.42 & 1.07 & 4.44 &1.11\\
             & 1/16& 1.58 & 1.07 & 2.59 & 1.13 & 5.49 & 1.24 &13.88 &1.42\\
(22+22)      & 1/32& 2.35 & 1.17 & 5.40 & 1.36 & 15.90& 1.77 &53.33 &2.73\\
             & 1/64& 4.25 & 1.39 &13.76 & 2.01 & 54.25& 3.89 &223.75&11.83\\
             &1/128& 9.19 & 1.94 &40.67 & 4.35 &195.60& 17.47&906.94&109.96\\
\noalign{\smallskip}\hline
\end{tabular}
\end{table}

We can see from Table \ref{table_1_remark} that normally the condition number $K(\textbf{L})$ is not very large, except when $\alpha$ tends to 2 and $h_1$, $h_2$ are much small, attributing to $h_1^{-\alpha}$ being comparably much larger than $h_2^{-\alpha}$ in $\textbf{D}$.

Since the coefficients ${g_k}$ are decreasing to zero from $g_2$ \cite{Meerschaert:04,Podlubny:99}, and so do ${w_k^{(\alpha)}}$ and ${\varpi_k^{(\alpha)}}$ in (\ref{equation2.4.6}) and (\ref{equation2.4.8}) for large $k$, also, the elements of $\textbf{A}_2$ in (\ref{equation:app.8}) are non-zero very $m$ columns in the first $(N_1+N_I)$ lines, when solving the fractional PDEs and well as ODEs for large $\alpha$ and small $h_1$, $h_2$, we can simply take the preconditioner $\textbf{P}$ of the matrix $\textbf{L}$ as
\begin{equation}\label{equation_8_remark}
\textbf{P}^{-1}=\sum_{k=-2}^{2}\textbf{L}_{k}+\sum_{k=3}^{2m}\textbf{L}_{-k},
\end{equation}
where $\textbf{L}_{k}$ is the matrix in which the entries outside the $k$-th diagonal are all zeros. Table \ref{table_1_remark} shows that by using this preconditioning procedure, the condition numbers of $\textbf{L}$ in several specific schemes can be effectively reduced for different $\alpha$ and $a=\epsilon=1$, $h_1=h_2/5$, and $\tau=1/400$.

\begin{remark}\label{remark3.2.3}
Since in the numerical section, the condition numbers of $\textbf{L}$ are tolerable, i.e., there is no difference between the numerical errors whether been preconditioned or not, we don't take the preconditioning procedure for the left matrix $\textbf{L}=\textbf{I}-\frac{\tau K}{2}\textbf{D}$.
\end{remark}




\end{document}